\documentclass[9pt,leqno,a4paper]{article}
\usepackage{amsfonts}
\usepackage{amsmath, amsfonts, amsthm, amssymb, amscd}
\usepackage[mathcal]{euscript}
\usepackage{mathrsfs}
\newtheorem{theorem}{Theorem}[section]

\newtheorem{lemma}[theorem]{Lemma}

\newtheorem{remark}[theorem]{Remark}
\numberwithin{equation}{section}

\usepackage{a4wide}

\newcommand \ndel {\not \hskip-.06cm  \partial} 

\newcommand \ih {\hat{i}}

\newcommand \jh {\hat{j}}
\newcommand \jc {\check{j}}

\newcommand \kh {\hat{k}}
\newcommand \kc {\check{k}}

\newcommand \lc {\check{l}}

\newcommand \delb {\overline{\del}}

\newcommand \delu {\underline{\del}}
\newcommand \Au {\underline{A}}
\newcommand \Bu {\underline{B}}
\newcommand \Pu {\underline{P}}

\newcommand \Boxt {\widetilde {\Box}}

\newcommand \del \partial

\newcommand \la \langle
\newcommand \ra \rangle

\let\oldmarginpar\marginpar
\renewcommand\marginpar[1]{\-\oldmarginpar[\raggedleft\footnotesize #1]%
{\raggedright\footnotesize #1}}

\begin{document}
\title{Global existence of small amplitude solutions to nonlinear coupled wave-Klein-Gordon systems in four space-time dimension with hyperboloidal foliation method}
\author{Yue MA\footnote{Laboratoir Jacques-Louis Lions. Email: ma@ann.jussieu.fr}}
\maketitle
\begin{abstract}
In this article one will develop a new type of energy method based on a foliation 
of space-time into hyperboloidal hypersurfaces . As we will see, with this method, some classical results such as global existence and 
almost global existence of regular solutions to the quasi-linear wave equations and Klein-Gordon 
equations will be established in a much simpler and much more natural way. Most importantly, the global existence 
of regular solutions to a general type of coupled quasilinear wave-Klein-Gordon system will be 
established. All of this suggests that compared withe the classical method, this hyperboloidal foliation of space-time may be a more natural way
to regard the wave operator.
\end{abstract}

\tableofcontents 
\section{Introduction}
In the research of quasilinear hyperbolic partial differential equations, the global existence 
of regular solutions is a central problem. There are already lots of excellent works. 
S. Klainerman has firstly developed the conformal Killing vector fields method. With this method, he has managed to 
establish the global existence of regular solution to quasilinear wave equations with classical null conditions
in $\mathbb{R}^{3+1}$ (see \cite{Kl1}) , and latter, global existence of regular solution to quasilinear Klein-Gordon equations 
in $\mathbb{R}^{3+1}$ (see \cite{Kl2}). From that time, this conformal Killing vector fields method has been developed and 
applied by many other to many more general cases.

However, because of an essential difference between wave equation and Klein-Gordon equation,
one of the conformal Killing vector field of wave equation, the scaling field $S$, is not a conformal Killing 
vector field of Klein-Gordon equation. More unfortunately, this scaling vector field plays un important role in the
decay estimates of wave equations. This leads to an essential difficulty when one attempt to establish the global
existence of coupled wave-Klein-Gordon system. However in \cite{Ka}, S. Katayama has established in a relatively special case 
the global in time existence of this kind of system with a technical $L^{\infty} - L^{\infty}$ type estimate. 

In another hand, L. H\"ormander has developed an ``alternative energy method" (see  \cite{Ho1}) for 
dealing the global existence of quasilinear Klein-Gordon equation. His observation is as follows. 
Consider the following Cauchy problem associated to the linear Klein-Gordon equation in $\mathbb{R}^{n+1}$
\begin{equation}\label{intro eq linear}
\left\{
\aligned
&\Box u + a^2u = f,
\\
&u(B+1,x) = u_0,\quad u_t(B+1,x) = u_1,
\endaligned
\right.
\end{equation}
where $u_0,\, u_1$ are regular functions supported on $\{(B+1,x):|x|\leq B\}$ and $f$ is also a regular 
function supported on 
$$
\Lambda':= \{(t,x): |x|\leq t-1\},
$$
with $a,B>0$ two fixed positive constants. 
By the Huygens' principle, the regular solution of \eqref{intro eq linear} is supported in 
$$
\Lambda' \cap \{t\geq B+1\}.
$$
One denotes by: 
$$
H_T := \{(t,x) : t^2 - x^2 = T^2,\, t>0\}
$$
and 
$$
G_{B+1} = \Lambda'\cap \{(t,x): \sqrt{t^2- x^2}\geq B+1\},
$$
one can develop a hyperboloidal foliation of $G_{B+1}$, which is
$$
G_{2B} = H_T \times [B+1, \, \infty).
$$
Then, taking $\del_t u$ as multiplier, the standard procedure of energy estimate leads one to the following
energy inequality
$$
E_m(H_T,\,u)^{1/2} \leq E_m(H_{B+1},\,u)^{1/2} + \int_{B+1}^T ds \bigg(\int_{H_s}f^2\bigg)^{1/2},
$$
where
$$
E_m(H_s,\,u) := \int_{H_s} \sum_{i=1}^{3}\big((x^i/t)\del_t u + \del_i u\big)^2 + ((T/t)\del_t u)^2 + (a/2)u^2\, dx .
$$
Then, H\"ormander has developed a Sobolev type estimate, see the lemma 7.6.1 of \cite{Ho1}. Combined with the 
energy estimate, he has managed to establish the decay estimate:
$$
\sup_{H_T} t^{n/2}|u| \leq \sum_{|I|\leq m_0} E_m(H_T,\,Z^I u)^{1/2}\leq E_m(H_{B+1},\,Z^I u)^{1/2} + \int_{B+1}^T ds \bigg(\int_{H_s}Z^I f^2\bigg)^{1/2},
$$
where $m_0$ is the smallest integer bigger the $n/2$.

But in the proof of \cite{Ho1}, the only used term of the energy $E_m(H_T,\,u)$ is the last term $u^2$. The first two terms seem to be
 omitted, at least when doing decay estimates. The new observation in this article is that the first two terms of the energy can also be used 
for estimating some important derivatives of the solution. This leads one to the possibility of applying this 
method on the case where $a=0$, which is the wave equation, so that the wave equations and the Klein-Gordon equations can 
be treated in the same framework. That is the key of dealing the coupled wave-Klein-Gordon system, and one may call it 
hyperboloidal foliation method.

The following is a prototype
of the main result which will be established in this article. Consider the Cauchy problem associated to the coupled wave-Klein-Gordon
system in $\mathbb{R}^{3+1}$:
\begin{equation}\label{intro eq wave-KG}
\left\{
\aligned
&\Box u = N(\del u,\del u) + Q_1(\del v,\del v) + Q_2(\del u,\del v),
\\
&\Box v + v = Q_3(\del u,\del u) + Q_4(\del v,\del v) + Q_5(\del u,\del v),
\\
&u(B+1,x) = \varepsilon u_0,\quad  u_t(B+1,x) = \varepsilon u_1,
\\
&v(B+1,x) = \varepsilon v_0,\quad  v_t(B+1,x) = \varepsilon v_1.
\endaligned
\right.
\end{equation}
Here the $N(\cdot,\,\cdot)$ are the classical null quadratic terms while $Q_i(\cdot,\cdot)$ are arbitrary quadratic terms. 
$u_0,\, u_1$ are regular functions supported on $\{(B+1,x):|x|\leq B\}$.
As we will see in the following, in general the global existence result holds
\begin{theorem}[Prototype of the main result]
There exists a $\varepsilon_0 > 0$ such that for any $0\leq \varepsilon \leq \varepsilon_0$, \eqref{intro eq wave-KG} has an unique 
global in time regular solution. 
\end{theorem}
This result has already been established in \cite{Ka} by S. Katayama. However, the advantages of the hyperboloidal foliation method are as follows.
First, it provides a proof much simpler than that of \cite{Ka} such that compared with the technical $L^{\infty}-L^{\infty}$ estimates,
it will use nothing else but the energy estimates and the Sobolev type inequalities in lemma 7.6.1 of \cite{Ho1}. Second, when one add terms such as
$\del_{\alpha}\del_{tt} w_j$ to the system, the method of \cite{Ka} does not work any more while the hyperbolic method still works.

Furthermore, this new method provides much simpler proofs when applied to many classical problems
such as the global existence of quasilinear wave equation in $\mathbb{R}^{3+1}$ with null conditions.

The structure of this article is as follows.
In section \ref{sec basic}, one will introduce the notation and establish the basic estimates. 
In section \ref{sec main}, one will establish the main result of this article, the global in time existence of regular solution to coupled quasilinear 
wave-Klein-Gordon system. 
\section{Notation and Basic estimates}\label{sec basic}
\subsection{Notation and general framework}
One denotes by $H_T$ the hyperboloid $\{t^2 - |x|^2 = T^2\}$ with radius $T>0$. Let
$$
\Lambda' = \{|x| \leq t - 1\},
$$                                                               
and 
$$
G^{T_2}_{T_1}=\{|x| \leq t-1,\,T_1\leq \sqrt{t^2 - |x|^2} \leq T_2\}.
$$ 
On $H_T \cap \Lambda'$, when $T\geq 1$ one has
$$
T\leq t \leq T^2,
$$
while on $H_T\cap \{r:=|x|\leq t/2\}$,
$$
T\leq t\leq \sqrt{2}T.
$$

Define the vector fields
$$
H_i = t\del_i + x^i\del_t.
$$
One sets $Z_{\alpha} : = \del_{\alpha}$ for $\alpha = 0,\ldots,3$, and $Z_{\alpha} := H_{\alpha-4}$ 
for $\alpha = 5,\,6,\,7$. One denotes also by $Z^J$ the $|J|$-th order 
operator $Z_{J_1}\cdots Z_{J_{|J|}}$, where $J$ is a multi-index 
with length $|J|$.
One notices that $H_j$ are tangent to $H_T$ and $\overline{\del_i} := t^{-1}H_i$ is the projection of $\del_i$ on $H_T$. 
Moreover if one uses $\{x^i\}$ as a coordinate system 
on $H_T$, then $\overline{\del_i}$ can also be regarded as the natural frame associated with these coordinates. 
on the tangent bundle on $H_T$, and the associated area element of $H_T$ is 
$$
d\sigma = t^{-1}\sqrt{t^2 + |x|^2}\,dx.
$$
One also defines the tangential derivatives $\ndel_i$,
$$
\ndel_i = \omega^i\del_t + \del_i,
$$ where $\omega^i := x^i/r$. Theses vector fields are tangent to the out-going light cone. 
\subsection{Basic energy estimates}
One considers the following linear wave or Klein-Gordon equation:
\begin{equation}\label{basic eq linear}
\left\{
\aligned
&\Box u + a^2u= f ,
\\
&u|_{H_{B+1}} = u_0,\quad \del_t u|_{H_{B+1}} = u_1,
\endaligned
\right.
\end{equation}
where $a$ is a nonnegative constant and $u_1,\, u_1$ are regular functions supported on $H_{B+1} \cap \Lambda'$. 
$f$ is also supposed to be a regular function with its support contained in $\Lambda'\cap G_{B+1}^{\infty}$.
Define the energy on the hyperboloid $H_T$:
\begin{equation}\label{basic energy expressions}
\aligned
E_m(T,\,u) 
&:= \int_{H_T}\bigg(|\del_t u|^2 + \sum_{i=1}^3|\del_i u|^2 + \frac{2x^i}{t}\del_tu \del_i u + 2(au)^2 \bigg) dx,
\\
&= \int_{H_T} 2(au)^2 + \sum_{i=1}^3 |\delb_i u|^2 + \bigg(\frac{T}{t}\del_t u  \bigg)^2 dx,
\\
&= \int_{H_T} 2(au)^2 + \sum_{i=1}^3 \bigg(\frac{T}{t} \del_i u\bigg)^2 + \sum_{i=1}^3 \bigg(\frac{r}{t}\del_i u + \frac{x^i}{r}\del_t u \bigg)^2 dx.
\endaligned
\end{equation}
\begin{proof}
See section 7.7 of \cite{Ho1}.
\end{proof}
Note that 
$$
\big|\big(\ndel_i - \delb_i\big)u\big| = \bigg|\frac{T\omega^i}{2t}\bigg|\bigg|\frac{T}{t}\del_t u\bigg|,
$$
which implies
$$
\int_{H_T}\sum_{i=1}^3|\ndel_i u|^2 dx \leq E(T,\,u).
$$
In general one has the following energy estimate:
\begin{lemma}[Energy estimates]
\label{basic ineq energy}
Let $u$ be a regular solution of \eqref{basic eq linear}
then the following energy estimate holds:
\begin{equation}\label{basic energy ineq trivial}
E_m(T,u) \leq E_m(B+1,u) + \int_{B+1}^{T}\bigg(\int_{H_s}f^2dx\bigg)^{1/2}ds.
\end{equation}
\end{lemma}

In Appendix A one will see that for the case treated in this article, the energy defined on hyperboloid is controlled by the standard energy. More 
precisely, for any $\varepsilon,C_1>0$, there exists an $\varepsilon'=\varepsilon'(B,\varepsilon)>0$ such that if 
$$
E^*(B+1,u)^{1/2} \leq \varepsilon',
$$
then
$$
E_m(B+1,u) \leq \varepsilon.
$$ 

\subsection{Commutators}
In this subsection one discusses some commutative relations. They are very important for establishing the decay estimates. Recall that $T^2=t^2-t^2$. 
The following commutative properties are obvious:
\begin{equation}\label{commutator H-wave}
[H_j,\,\Box] = 0,
\end{equation}
and
\begin{equation}\label{commutator partial-wave}
[\del_{\alpha},\,\Box] = 0.
\end{equation}
One defines the vector field family
$$
\mathscr{D}_g:= \{\ndel_i,\,(t-r)r^{-1}\del_{\alpha}\}.
$$
Firstly one has the following commutative relations between $H_j$ and $\del_{\alpha}$:
\begin{equation}\label{commutator base H-partial}
\aligned
&H_j\del_t u = \del_t H_j u - \del_j u,\\
&H_j\del_i u = \del_i H_j u - \delta^i_j\del_t u.
\endaligned
\end{equation}
Notice that one has 
$$
H_j\bigg(\frac{T}{t}\bigg) = -\frac{x^jT}{t^2},
$$
so one gets the following commutative relations between $H_j$ and $\frac{T}{t}\del_{\alpha}$:
\begin{equation}\label{commutator base H-T/t}
\aligned
&H_j\bigg(\frac{T}{t}\del_t u\bigg) = -\frac{T}{t}\bigg(\del_j u + \frac{x^j}{t}\del_t u\bigg) + \frac{T}{t}\del_t(H_j u),
\\
&H_j\bigg(\frac{T}{t}\del_i u\bigg) = -\frac{T}{t}\bigg(\delta_i^j\del_t u + \frac{x^j}{t}\del_i u\bigg) + \frac{T}{t}\del_t(H_j u).
\endaligned
\end{equation}
The commutative relations between $H_j$ and $D_g\in \mathscr{D}_g$ are
\begin{equation}\label{commutator base H-tangental}
\aligned
&H_j\ndel_iu = \ndel_iH_ju - \omega^i\ndel_ju + \big(\delta^i_j - \omega^i\omega^j\big)\bigg(\frac{t}{r}-1\bigg)\del_t,
\\
&H_j\bigg(\Big(\frac{t}{r} - 1\Big)\del_t u \bigg) = \bigg(\frac{t}{r} -1 \bigg)\del_t(H_j u) -\omega^j\bigg(\frac{t}{r} + 1\bigg)\bigg(\frac{t}{r} -1 \bigg)\del_t u
- \bigg(\frac{t}{r} -1 \bigg)\del_j u,
\\
&H_j\bigg(\Big(\frac{t}{r} - 1\Big)\del_i u \bigg) =  \bigg(\frac{t}{r} -1 \bigg) \del_i(H_j u)  -\omega^j\bigg(\frac{t}{r} + 1\bigg)\bigg(\frac{t}{r} -1 \bigg)\del_i u
- \delta_j^i\bigg(\frac{t}{r} -1 \bigg)\del_t u.
\endaligned
\end{equation}
The commutative properties between $H_j$ and $\delb_i$ are
\begin{equation}\label{commutator base H-par_b}
H_j\delb_iu = \delb_j H_iu - \frac{x^j}{t}\delb_j u.
\end{equation}
In general one has the following results:
\begin{lemma}\label{basic lem commutator H}
Let $I$ be a arbitrary multi-index. In the region $\Lambda'\cap G_{B+1}^{\infty}$ one has:
\begin{equation}\label{commutator H-partial}
|H^I \del_{\alpha} u| \leq |\del_{\alpha} H^I u| + C(n,|I|)\sum_{|J|<|I|}\sum_{\beta=0}^n |\del_{\beta} H^J u|,
\end{equation}
\begin{equation}\label{commutator H-par_b}
\sum_{i=1}^3\sum_{|I|=p} |H^I \delb_i u| \leq \sum_{i=1}^3\sum_{|I|=p} |\delb_i H^I u| + C(n,|I|)\sum_{|J|<p}\sum_{j=1}^n |\delb_j H^J u|,
\end{equation}
and
\begin{equation}\label{commutator H-T/t}
\bigg|H^I \bigg(\frac{T}{t}\del_{\alpha} u\bigg)\bigg| \leq \bigg|\frac{T}{t}\del_{\alpha} H^I u\bigg| 
+ C(n,|I|)\sum_{|J|<|I|}\sum_{\beta=0}^3 \bigg|\frac{T}{t}\del_{\beta} H^J u\bigg|.
\end{equation}
In the region $\{r \geq t/2\} \cap \Lambda'\cap G_{B+1}^{\infty}$ one has:
\begin{equation}\label{commutator H-tangental}
|H^I D_g u| \leq |D_g H^I u| + C(n,|I|)\sum_{|J|< |I|\atop Y\in \mathscr{D}_g}|Y H^J u|, 
\end{equation}
where $C(n,|I|)$ is a constant depending only on dimension $n$ and $|I|$.
\end{lemma}
\begin{proof}
The inequality \eqref{commutator H-partial} is a direct result of \eqref{commutator base H-partial}.
To prove the \eqref{commutator H-tangental}, one notices that in \eqref{commutator base H-tangental}, one has 
some non-constant coefficients. To get the result, one calculates the their derivatives, 
$$
\aligned
&H_j\big(\omega^i\big) = \big(\delta^i_j - \omega^i\omega^j\big)\frac{t}{r},\\
&H_j\bigg(\frac{t}{r}\bigg) = \frac{x^j}{r} - \frac{t^2x^j}{r^3} = -\omega^j\bigg(\frac{t}{r} + 1\bigg)\bigg(\frac{t}{r} -1\bigg),
\\
&H_j\bigg(\frac{t}{r} +1\bigg) = \frac{x^j}{r} - \frac{t^2x^j}{r^3} = -\omega^j\bigg(\frac{t}{r} -1\bigg)\bigg(\frac{t}{r} + 1\bigg),
\\
&H_j\bigg(\frac{t}{r} -1\bigg) = \frac{x^j}{r} - \frac{t^2x^j}{r^3} = -\omega^j\bigg(\frac{t}{r} + 1\bigg)\bigg(\frac{t}{r} -1\bigg),
\endaligned
$$
One denotes by $\mathscr{F}$ the family of functions:
$$
\mathscr{F} := \{\omega,\,\frac{t}{r},\, \frac{t}{r} + 1,\, \frac{t}{r} - 1\}. 
$$
Then one gets, for arbitrary multi-index, 
$$
H^I D_g u = D_g H^I u + \sum_{|J|<|I|} F^{K(I,J)}D_g H^J u,
$$
where $K(I,J)$ is a multi-index of length $|K|$ depending only on two other multi-index $I$ and $J$, and 
$F^K$ is a multiplier defined as follows:
$$
F^K := F^{K_1}F^{K_2}\cdots F^{K_{|K|}}, \quad F^{K_i}\in \mathscr{F}.
$$
Notice that when $t/2\leq r \leq t-1$, all the terms on the right-hand-side are bounded, which gives \eqref{commutator H-tangental}.

The inequalities \eqref{commutator H-par_b} and \eqref{commutator H-T/t} are proved similarly, one omits the details.
\end{proof}

One has also the following commutative relations between $\del_{\alpha}$ and $D_g$:
\begin{equation}\label{commutator del-good}
\aligned
&\del_j\ndel_i = \big(\delta^i_j r^{-1} - \omega^i\omega^j r^{-1}\big)\del_t + \ndel_i\del_j,\\
&\del_t\ndel_i = \ndel_i \del_t,\\
&\del_j\bigg(\frac{t}{r} - 1\bigg)\del_t = -\omega^j\frac{t}{r^2}\del_t + \bigg(\frac{t}{r} - 1\bigg)\del_t\del_j,\\
&\del_t\bigg(\frac{t}{r} - 1\bigg)\del_t = \frac{1}{r}\del_t + \bigg(\frac{t}{r} - 1\bigg)\del_t \del_t,\\
&\del_j\bigg(\frac{t}{r} - 1\bigg)\del_i = -\omega^j\frac{t}{r^2}\del_i + \bigg(\frac{t}{r} - 1\bigg)\del_i\del_j,\\
&\del_t\bigg(\frac{t}{r} - 1\bigg)\del_i = \frac{1}{r}\del_i + \bigg(\frac{t}{r} - 1\bigg)\del_i\del_t.
\endaligned
\end{equation}
And, the following result:
\begin{lemma}\label{basic lem commutator}
Let $u$ be a regular function on $\Lambda'\cap G_{B+1}^{\infty}$. 
The the following estimates are true
$$
\sum_{|J|=p}\big|Z^J\delb_i u\big| \leq C(p,n)\sum_{|\beta|\leq p} \sum_{j=1}^3 \big|\delb_j Z^{\beta}u\big|,
$$
$$
\big|Z^I\del_{\alpha} u\big| \leq \big|\del_{\alpha}Z^I u\big| + C(|I|,n)\sum_{|J|<|I|} \sum_{\alpha=0}^3 \big|\del_{\alpha} Z^J u\big|.
$$
When $r \geq \frac{1}{2}t$, then the following estimate is true
$$
\big|Z^I\ndel_i u\big| \leq \big|\ndel_i Z^I u\big| + C(|I|,n)\sum_{|J|<|I|}\sum_{j=1}^3\big|\ndel_j Z^J u\big|
+ C(|I|,n)(T^2/t^2)\sum_{|J|<|I|}\sum_{\alpha=0}^3\big|\del_{\alpha}Z^I u\big|.
$$
\end{lemma}
\begin{proof}
Considering the commutative relation \eqref{commutator del-good}, the proof is the same to that of lemma \ref{basic lem commutator H}.
\end{proof}
\subsection{Frame and the null conditions}
In this subsection, a so called ``one-one" frame, denoted by $\{\delu_{\alpha}\}$,will be introduced. Here 
$$
\aligned
&\delu_0 := \del_t,
\\
&\delu_i := \ndel_i, \quad i=1,2,3. 
\endaligned
$$
The transition matrix between this frame and the natural frame is 
$$
\Phi:=
\begin{pmatrix}
&1        &0         &0         &0        \\
&\omega^1 &1         &0         &0        \\
&\omega^2 &0         &1         &0        \\
&\omega^3 &0         &0         &1        \\
\end{pmatrix},
$$
so that 
$$
\delu_{\alpha}u = \Phi_{\alpha}^{\beta}\del_{\beta}u.
$$
Its inverse is 
$$
\Psi := \Phi^{-1} =
\begin{pmatrix}
&1         &0         &0         &0        \\
&-\omega^1 &1         &0         &0        \\
&-\omega^2 &0         &1         &0        \\
&-\omega^3 &0         &0         &1        \\
\end{pmatrix},
$$ 
so that
$$
\del_{\alpha}u = \Psi_{\alpha}^{\beta}\delu_{\beta}u.
$$
The advantage of this ``one-one" frame is that the last three vector fields are tangent to the outgoing light cone. In these directions, the gradient
of the solution has better decay near the light cone.
For a general two tenser $\mathcal{T}$, one can write it in the nature frame as 
$$
\mathcal{T} = T^{\alpha\beta}\del_{\alpha}\del_{\beta},
$$ 
or in the ``one-one" frame 
$$
\mathcal{T} = \underline{T}^{\alpha\beta}\delu_{\alpha}\delu_{\beta},
$$
In general the following result holds:
\begin{lemma}\label{basic lem frame}
For any two tenser $\mathcal{T}$ one has, in the region $\Lambda'\cap \{r\geq t/2\}$,
$$
\sum_{\alpha,\beta=0}^3|Z^I \underline{T}^{\alpha\beta}| \leq C(|I|,n)\sum_{|J|\leq|I|}\sum_{\alpha',\beta'=0}^3 |Z^J T^{\alpha'\beta'}|.
$$
\end{lemma}
\begin{proof}
Note that in the region $\{r\geq t/2\}\cap \Lambda'$,
$$
Z \omega \leq 1.
$$
Then the proof is just a simple calculation.
\end{proof}

Now it is the time to introduce the classical null conditions. Let $\{f_i\}$ be a finite set of regular function on $\mathbb{R}^{n+1}$. 
The following quadratic forms
$$
A^{\alpha\beta\gamma j}_i \del_{\gamma}f_j \del_{\alpha\beta}f_i,
\quad
B^{\alpha\beta j}_i f_j \del_{\alpha\beta}f_i,
\quad
P^{\alpha\beta i j}\del_{\alpha}f_i\del_{\beta}f_j
$$
are said to satisfy the null conditions if for any $\xi\in \mathbb{R}^{n+1}$ such that $\xi_0\xi_0 - \sum_{i=1}^3\xi_i\xi_i = 0$, 
\begin{equation}\label{basic null conditions}
A^{\alpha\beta\gamma j}_i\xi_{\alpha}\xi_{\beta}\xi_{\gamma} = 
B^{\alpha\beta j}_i\xi_{\alpha}\xi_{\beta} = P^{\alpha\beta ij} \xi_{\alpha}\xi_{\beta} = 0.
\end{equation}
Clearly the following conditions are weaker than these null conditions:
\begin{equation}\label{basic weak null conditions}
\Au^{000j}_i = \Bu^{00 j}_i = \Pu^{00 ij} = 0.
\end{equation}
\subsection{Basic decay estimates}
For the convenience of statement, one defines the following norm:
$$
||u||^2_{H,p,H_T} := \sum_{|I|\leq p}\int_{H_T} |H^I u|^2 dx.
$$
To turn $L^2$ estimates into $L^{\infty}$ estimates, one needs the following Sobolev inequality, 
which is a slightly improvement of a result by H\"ormander (see lemma 7.6.1 of\cite{Ho1}).
\begin{lemma}[Sobolev-type estimate on hyperboloid] 
\label{basic lem sobolev}
Let $p(n)$ be the smallest integer $>n/2$. Any $C^{\infty}$ function defined on $\mathbb{R}^{1+n}$ 
satisfies 
\begin{equation}
\label{inequ sobolev_on_Hyper}
\sup_{H_T}t^n|u(t,x)|^2 \leq C(n) ||u||^2_{H,p(n),H_T},
\end{equation}
where $C(n)>0$ is a constant depending only on dimension $n$. 
\end{lemma}
\begin{proof}
One observes that the derivatives $\del_{\alpha}$ are not actually used.
\end{proof}

Now it is the time to establish the basic decay estimates.
\begin{lemma}\label{basic lm dispersive}
Suppose $u$ is the regular solution of \eqref{basic eq linear}, then $u$ satisfies the following decay estimates:
\begin{equation}
\sup_{H_T}t^n\big(|\delb_i u|^2 + |(T/t)\del_{\alpha} u|^2 + |au|^2\big) + \sup_{H_T\cap\{r\geq t/2\}}t^n|\ndel_i u|^2 \leq\sum_{|I|\leq p(n)}C(n)E_m(T, H^I u)
\end{equation}
\end{lemma}
\begin{proof}
The proof is just a combination of energy estimat\eqref{basic energy ineq trivial}, the commutation estimate \eqref{basic lem commutator H} and the Sobolev
inequality \eqref{basic lem sobolev}.
\end{proof}
\begin{remark}
From lemma \ref{basic lm dispersive}, one easily gets the following result. Taking $f=0$ and $g^{\alpha\beta} = m^{\alpha\beta}$, then the solution
of homogeneous linear wave equation has decay rate as 
$$
|\del_{\alpha} u| \leq C t^{1-n/2}T^{-1} = C t^{-(n-1)/2}(t-r+1)^{-1/2},
$$  
which is exactly the classical result.
\end{remark}

\section{Main result}\label{sec main}
\subsection{Formalization of the problem and statement of the main result}
One considers the following Cauchy problem associated to the quasilinear system:
\begin{equation}\label{quasilinear eq wave-KG nonlinear}
\left\{
\aligned
&\Box w_i + G_i^{j\alpha\beta}(w,\del w)\del_{\alpha\beta}w_j + D_i^2 w_i= F_i(w,\,\del w),\\ 
&w_i(B_0+1,x) = \varepsilon' {w_i}_0,\\
&\del_t w_i(B_0+1,x) = \varepsilon' {w_i}_1,
\endaligned
\right.
\end{equation}
with
$D_i^2$ constants, $D_i = 0$ for $1\leq i \leq j_0$ and $D_i>0$ for $j_0+1 \leq j\leq j_0+k_0$. For simplicity one suppose 
that $D_i \geq 1$ for $j_0+1 \leq i \leq j_0 + k_0$,
and
$$
G^{j\alpha\beta}_i(w,\,\del w) = A^{j\alpha\beta\gamma k}_i\del_{\gamma} w_k + B^{j\alpha\beta k}_i w_k + O(|w| + |w'|)^2,
$$ 
and 
$$
F_i(w,\,\del w) = P^{\alpha\beta jk}_i\del_{\alpha} w_j \del_{\beta} w_k +Q^{\alpha jk}_i w_j\del_{\alpha}w_k + R_i^{jk}w_jw_k + O(|w| + |w'|)^3.
$$
Here $m^{\alpha\beta}$ are the coefficients of wave operator, and all $A^{j\alpha\beta\gamma k}_i,B^{j\alpha\beta k}_i,P^{\alpha\beta jk}_i,Q^{\alpha jk}_i$ 
and $R_i^{jk}$ are constants with the absolute value controlled by $K$. Without lose of generality, one supposes 
$G_i^{j\alpha\beta}=G_i^{j\beta\alpha}$. To insure the hyperbolicity of the system, the following conditions of symmetry are imposed:
\begin{equation}\label{quasilinear conditions of symmetry}
G_i^{j\alpha\beta} = G_j^{i\beta\alpha}.
\end{equation}

For the convenience of proof, one makes the following convention of index:
the Latin index $i,j,k,l,...$ denote the positive entire number $1,2,3,..k_0 + j_0$. The Greek index $\alpha,\beta,\gamma,...$ denote the nonnegative 
entire number $0,1,2,3$. The Latin index with a circumflex accent above it such as $\jh$ denote the entire number $1,2,...j_0$ and the Latin index with
a hacek on it such as $\jc$ denote the entire number $j_0+1, j_0+2,...j_0+k_0$.  One also denote by
$$
u_{\jh} := w_{\jh},
\quad
v_{\kc} := w_{\kc},
$$
the different components of $w_j$.

One supposes that 
\begin{equation}\label{quasilinear decay on coefficients}
B_i^{j\alpha\beta \kh} = Q_i^{\alpha j\kh} = R_i^{\jh\kc} = R_i^{\jh\kh}=0.
\end{equation}

One also suppose the weak null conditions:
\begin{equation}\label{quasilinear null conditions}
\Au_{\ih}^{\jh000\kh}=\Bu_{\ih}^{\jh00\kh}=\Pu_{\ih}^{00\jh\kh}=0,
\end{equation}
The initial data ${w_i}_0, {w_i}_1$ are supposed to be regular functions compactly supported on the disc $|x|\leq B$.
Then in general the following global-in-time existence holds
\begin{theorem}\label{quasilinear thm main}
Suppose 
\eqref{quasilinear null conditions} holds.
Then there exists an $\varepsilon_0>0$ such that for any $0\leq \varepsilon' \leq \varepsilon_0$, the cauchy problem 
\eqref{quasilinear eq wave-KG nonlinear} has a unique global in time solution.
\end{theorem}
\begin{remark}
\begin{itemize}
\item 
One improvement compared with \cite{Ka} is that in theorem \ref{quasilinear thm main}, the nonlinear part can have a term such as 
$\del_{\alpha}w_i\del_{tt}^2 w_i$. Further more, in the proof, the technical $L^{\infty}-L^{\infty}$ estimate will not be used.
\item
This result is also valid for the case where the initial data ${w_i}_0\in H^7(\mathbb{R}^3)$ and ${w_i}_0\in H^6(\mathbb{R}^3)$.
Compared with \cite{Ka}, where one needs ${w_i}_0\in H^{19}(\mathbb{R}^3)$ and ${w_i}_1\in H^{18}(\mathbb{R}^3)$, this is also a improvement. 
\item
The condition \eqref{quasilinear null conditions} and \eqref{quasilinear decay on coefficients} are  far from optimal. In general this method of proof
can be applied to the case where the coefficients $A^{j\alpha\beta\gamma k}_i,B^{j\alpha\beta k}_i,P^{\alpha\beta jk}_i,Q^{\alpha jk}_i$ 
and $R_i^{jk}$ are regular functions with certain increasing rates . But here one prefers to write a some-how restricted but short theorem. 
Readers may check the proof of 
lemma \ref{quasilinear lem curved energy is big} 
and \ref{quasilinear lem source u}, \ref{quasilinear lem source-commutator}, \ref{quasilinear lem energy curveterm is small} to get weaker assumptions on coefficients.
\end{itemize}
\end{remark}
\subsection{Preparations}
To prove this result we need some preparations.
The following lemma is the principle of the so called bootstrap method:
\begin{lemma}[Principle of continuity]\label{basic lem prcp of cont.}
Let $u$ be the regular local in time solution of the following quasilinear wave equation
\begin{equation}
\left\{
\aligned
&\Boxt_{g(u,\del u)} u = F(u,\del u),\\
&u(0,x)=u_0,\quad u_t(0,x) = u_1.
\endaligned 
\right.
\end{equation}
Then if the life span time $T^*$ of $u$ is finite, one has
\begin{equation}\label{basic lem prcp of cont. condition}
\sup_{0\leq s < T^*\atop |I|\leq p(n),|J|\leq 2} E_m(s,H^I \del^J u) = \infty.
\end{equation}
\end{lemma}
\begin{proof}
If \eqref{basic lem prcp of cont. condition} does not hold then from lemma \ref{basic lem sobolev}
$$
\sup_{0\leq s < T^*\atop |J|\leq 2}|\del^J u| < \infty
$$
By the theorem 6.4.11 of \cite{Ho1}, one sees that $T^* = \infty$.
\end{proof}

A second result that one needs is a more technical energy estimated will be established, which is designed for the proof. 
One defines the following curved energy on hyperboloid $H_s$:
$$
E_G(s,w_i) := 
E_m(s,w_i) + 2\int_{H_s} \big(\del_t w_i \del_{\beta}w_j G_i^{j\alpha\beta}\big) \cdot (1,-x^a/t) dx 
- \int_{H_s} \big(\del_{\alpha}w_i\del_{\beta}w_j G_i^{j\alpha\beta}\big)dx.
$$
The the following energy estimate holds:
\begin{lemma}[Energy estimate]\label{quasilinear lem energy estimate}
Let 
$\{w_i\}$ a regular solution to the Cauchy problem \eqref{quasilinear eq wave-KG nonlinear}. Suppose that the following estimates holds:
If the following assumptions holds,
\begin{equation}\label{quasilinear curved energy is big}
\sum_iE_m(s,w_i)\leq  3\sum_iE_G(s,w_i),
\end{equation}
\begin{equation}\label{quasilinear energy curveterm is small}
\int_{H_s}\bigg(\del_{\alpha}G_i^{j\alpha\beta}\del_t w_i\del_{\beta}w_j-\frac{1}{2}\del_t G_i^{j\alpha\beta}\del_{\alpha}w_i\del_{\beta}w_j \bigg)\frac{s}{t} dx 
\leq M_i(s) \sum_iE_m(s,w_i),
\end{equation}
and
\begin{equation}\label{quasilinear energy source}
\bigg(\int_{H_s} |F_i|^2dx\bigg)^{1/2} \leq L_i(s) + N_i(s)\sum_j E_m(s,w_j)^{1/2}.
\end{equation}
Then the following energy estimate holds:
\begin{equation}\label{quasilinear esti energy}
\aligned
&\quad\bigg(\sum_i E_m(s,w_i)\bigg)^{1/2} 
\\
&\leq \sqrt{3}\bigg(\sum_i E_G(B+1,w_i)\bigg)^{1/2}\exp\bigg(\int_{B+1}^s \sum_i\Big(3M_i(\tau)+\sqrt{3(j_0+k_0)}N_i(\tau)\Big)d\tau\bigg)
\\
&\quad +\int_{B+1}^s3\sum_iL_i(\tau)\exp\bigg(\int_{B+1}^\tau \sum_i\Big(3M_i(\tau')+\sqrt{3(j_0+k_0)}N_i(\tau')d\tau'\Big)\bigg)d\tau.
\endaligned
\end{equation} 
\end{lemma}

\begin{proof}
Under the assumptions \eqref{quasilinear conditions of symmetry},
taking $\del_t w_i$ as multiplier, the standard energy estimate procedure gives
$$
\aligned
&\sum_i\bigg(\frac{1}{2} \del_t \sum_{\alpha}\big(\del_{\alpha} w_i\big)^2 + \sum_a\del_a\big(\del_a w_i \del_t w_i\big)
 + \del_{\alpha}\big(G_i^{j\alpha\beta}\del_t w_i\del_{\beta}w_j\big) 
- \frac{1}{2}\del_t\big(G_i^{j\alpha\beta}\del_{\alpha}w_i\del_{\beta}w_jG\big)\bigg)
\\
& = \sum_i \del_t w_i F_i 
+ \sum_i\bigg(\del_{\alpha}G_i^{j\alpha\beta} \del_tw_i \del_{\beta}w_j - \frac{1}{2}\del_tG_i^{j\alpha\beta}\del_{\alpha}w_i\del_{\beta}w_j\bigg)
\endaligned
$$
Then integrate in the region $G_{B+1}^s$ and use the Stokes formulae, one gets
$$
\aligned
\frac{1}{2}\sum_i \big(E_G(s,w_i) -E_G(B,w_i)\big) 
&= \int_{G_{B+1}^s} \del_t w_i F_i dx + \sum_i\int_{G_{B+1}^s}\del_{\alpha}G_i^{j\alpha\beta} \del_tw_i \del_{\beta}w_j 
\\
&\quad - \frac{1}{2}\del_tG_i^{j\alpha\beta}\del_{\alpha}w_i\del_{\beta}w_j\,dx,
\endaligned
$$
which leads to
$$
\aligned
\quad \frac{d}{ds}\sum_i E_G(s,w_i) 
&= 2\sum_i\int_{H_s}(s/t)\del_{\alpha}G_i^{j\alpha\beta} \del_tw_i \del_{\beta}w_j -(s/2t)\del_tG_i^{j\alpha\beta}\del_{\alpha}w_i\del_{\beta}w_j\,dx
\\
&\quad +2\int_{H_s}(s/t) \del_t w_i F_i dx 
\endaligned
$$
So one gets
$$
\aligned
&\quad\bigg(\sum_i E_G(s,w_i)\bigg)^{1/2}\frac{d}{ds}\bigg(\sum_i E_G(s,w_i)\bigg)^{1/2}
\\
&\leq \sqrt{3}\sum_i\bigg(\int_{H_s}\big|F_i\big|^2dx\bigg)^{1/2} E_m(s,w_i)^{1/2} + \sum_i M_i(s)\sum_jE_m(s,w_j)
\\
&\leq \sqrt{3}\bigg(\sum_i\int_{H_s}\big|F_i\big|^2dx\bigg)^{1/2} \bigg(\sum_iE_G(s,w_i)\bigg)^{1/2} + 3\sum_i M_i(s)\sum_jE_G(s,w_j),
\endaligned
$$
which leads to 
$$
\aligned
&\quad\frac{d}{ds}\bigg(\sum_i E_G(s,w_i)\bigg)^{1/2} 
\\
&\leq \sqrt{3}\bigg(\sum_i\int_{H_s}\big|F_i\big|^2dx\bigg)^{1/2} + 3\sum_i M_i(s)\bigg(\sum_jE_G(s,w_j)\bigg)^{1/2}
\\
&\leq \sqrt{3}\sum_i\bigg(L_i(s) + N_i(s)\sum_j E_G(s,w_j)^{1/2}\bigg) + 3\sum_i M_i(s)\bigg(\sum_jE_G(s,w_j)\bigg)^{1/2}
\\
&\leq \sqrt{3}\sum_i L_i(s) + \sqrt{3(j_0+k_0)} \sum_i N_i(s)\bigg(\sum_j E_G(s,w_j)\bigg)^{1/2} + 3\sum_i M_i(s)\bigg(\sum_jE_G(s,w_j)\bigg)^{1/2}.
\endaligned
$$
By Gronwall's lemma, \eqref{quasilinear esti energy} is proved.
\end{proof}
\subsection{Proof of the main result}
\begin{proof}[proof of theorem \ref{quasilinear thm main}]
For any  $\varepsilon,C_1 > 0$, by theorem \ref{appendix Thm A}, there exists an $\varepsilon_0(B)>0$ such that for any 
$0\leq \varepsilon' \leq \varepsilon_0(B)$, $E_m(B+1,Z^I w_i) \leq \varepsilon C_1$.
Then one uses the continuity method. Suppose that on a interval $[B+1,T)$, the energy $E_m(s,Z^Iw_i)$  satisfy 
\begin{equation}\label{quasilinear proof energy assumption 0}
\aligned
&\sum_{|I^*|\leq 7}\bigg(\sum_i E_m(s,Z^I w_i)\bigg)^{1/2} \leq C_1\varepsilon s^{\delta},
\\
&E_m(s,Z^I u_{\ih}) \leq C_1\varepsilon,\quad \text{for}\quad 1\leq \ih\leq j_0. \quad |I|\leq 5.
\endaligned
\end{equation}
where $0< \delta \leq 1/6$.
By lemma \ref{basic lem commutator}, one has the following $L^2$ estimates:
\begin{equation}\label{quasilinear proof energy assumption}
\aligned
&\sum_{\ih,\alpha\atop |I|\leq 5}\bigg(\int_{H_s}\big|(s/t)Z^I \del_{\alpha} u_{\ih}\big|^2 dx\bigg)^{1/2} 
+ \sum_{\ih,a\atop |I|\leq 5}\bigg(\int_{H_s\cap\{r\geq r/2\}}\big| Z^I \delb_a u_{\ih}\big|^2 dx\bigg)^{1/2} \leq C(n) C_1 \varepsilon,
\\
&\sum_{\ih,\alpha\atop|I^*|\leq 7}\bigg(\int_{H_s}\big|(s/t)Z^{I^*} \del_{\alpha} w_i\big|^2 dx\bigg)^{1/2}
+ \sum_{\ih,a\atop |I^*|\leq 7}\bigg(\int_{H_s\cap\{r\geq r/2\}}\big| Z^{I^*} \delb_a w_{\ih}\big|^2 dx\bigg)^{1/2} \leq C(n) C_1 \varepsilon s^{\delta},
\\
&\sum_{\jc\atop|I^*|\leq 7}\bigg(\int_{H_s}\big|Z^{I^*}v_{\jc}\big|^2dx\bigg)^{1/2} \leq C(n) C_1 \varepsilon s^{\delta}.
\endaligned
\end{equation}
Also, by lemma \ref{basic lm dispersive} and lemma \ref{basic lem decay on u}, one has, for $|J|\leq 3$ and $|J^*|\leq 5$,
\begin{equation}\label{quasilinear proof basic decay}
\aligned
&\sup_{H_s}\big|s t^{1/2} \del_{\alpha}Z^J u_{\jh}\big| + \sup_{H_s\cap \{r\geq t/2\}}\big|t^{3/2}\ndel_a Z^J u_{\jh}\big| \leq C(n)C_1 \varepsilon,
\\
&\sup_{H_s}\Big(\big|s t^{1/2} \del_{\alpha}Z^{J^*} v_{\kc}\big|\ + \big|t^{3/2}Z^{J^*}v_{\kc}\big| \Big) 
+ \sup_{H_s\cap \{r\geq t/2\}}\big|t^{3/2}\ndel_a Z^{J^*} v_{\kc}\big|
\leq C(n)C_1\varepsilon s^{\delta}.
\endaligned
\end{equation}
From lemma \ref{basic lem commutator}, one gets
\begin{equation}\label{quasilinear proof used decay}
\aligned
&\sup_{H_s}\big|s t^{1/2} Z^J \del_{\alpha}u_{\ih}\big| + \sup_{H_s\cap \{r\geq t/2\}}\big|t^{3/2} Z^J \ndel_a u_{\ih}\big| \leq C(n)C_1 \varepsilon,
\\
&\sup_{H_s}\Big(\big|s t^{1/2} Z^{J^*}\del_{\alpha} v_{\kc}\big|\ + \big|t^{3/2}Z^{J^*}v_{\kc}\big| \Big) 
+ \sup_{H_s\cap \{r\geq t/2\}}\big|t^{3/2}Z^{J^*} \ndel_a v_{\kc}\big|
\leq C(n)C_1\varepsilon s^{\delta}.
\endaligned
\end{equation}
One derives the equation \eqref{quasilinear eq wave-KG nonlinear} with respect to a product $Z^I$, and gets
\begin{equation}\label{quasilinear proof derived eq 1}
\Box Z^Iw_i + G_i^{j\alpha\beta}\del_{\alpha\beta}Z^Iw_j +  D_i^2Z^I w_i = -[Z^I,G_i^{j\alpha\beta}\del_{\alpha\beta}]w_j + Z^IF_i(w,w').
\end{equation}
One also writes the first $j_0$ equations:
\begin{equation}\label{quasilinear proof derived eq 2}
\Box Z^I u_{\jh} = Z^I F_{\jh} - Z^I\big(G_{\jh}^{k\alpha\beta}\del_{\alpha\beta}w_k\big).
\end{equation}
For technical reason, when $|I|=7$ and $|J|=6$,  the following system will also be considered. 
\begin{equation}\label{quasilinear proof derived eq 3}
\left\{
\aligned
&\Box Z^Iw_i + G_i^{j\alpha\beta}\del_{\alpha\beta}Z^I w_j +  D_i^2Z^I w_i = -[Z^I,G_i^{j\alpha\beta}\del_{\alpha\beta}]w_j + Z^IF_i(w,w').
\\
&\Box Z^Jw_i + G_i^{j\alpha\beta}\del_{\alpha\beta}Z^J w_j +  D_i^2Z^I w_i = -[Z^J,G_i^{j\alpha\beta}\del_{\alpha\beta}]w_j + Z^JF_i(w,w').
\endaligned
\right.
\end{equation}
Then
by \eqref{quasilinear proof derived eq 1}, the energy estimate \eqref{basic energy ineq trivial} gives for any $|I|\leq 5$ and $1\leq \jh\leq j_0$,
$$
E_m(s,Z^Iu_{\ih})^{1/2} 
\leq E_m(s,Z^I u_{\ih})^{1/2} + \int_{B+1}^s \bigg(\int_{H_\tau}\big|Z^IF_{\ih} - Z^I\big(G_{\ih}^{j\alpha\beta}w_j\big)\big|^2dx\bigg)^{1/2}d\tau.
$$

By \eqref{quasilinear proof derived eq 3}, the energy estimate \eqref{quasilinear esti energy} gives,
$$
\aligned
&\quad\bigg(\sum_{i\atop 6\leq |I^*|\leq 7} E_m(s,Z^{I^*}w_i)\bigg)^{1/2}
\\
&\leq \sqrt{3}\bigg(\sum_{i\atop 6\leq |I^*|\leq7} E_m(B+1,Z^{I^*}w_i)\bigg)^{1/2}\exp\bigg(\int_{B+1}^s\sum_i\big(3M_i(\tau)+ \sqrt{3(j_0+k_0)}N_i(\tau)\big)d\tau\bigg)
\\
&\quad+\int_{B+1}^s3\sum_iL_i(\tau)\exp\bigg(\int_{B+1}^s\sum_i\big(3M_i(\tau')+ \sqrt{3(j_0+k_0)}N_i(\tau')\big)d\tau'\bigg)d\tau.
\endaligned
$$
where
$$
\bigg(\int_{H_s}\big|Z^{I^*}F_i - [Z^{I^*}, G_i^{j\alpha\beta}\del_{\alpha\beta}]w_j \big|^2dx\bigg) 
\leq L_i(s) + N_i(s)\sum_{i\atop 6\leq|I|\leq7}E_m(s,Z^I w_i)^{1/2}
$$
and
$$
\int_{H_s}\frac{s}{t}\bigg(\big(\del_{\alpha}G_i^{j\alpha\beta}\big)\del_t Z^{I^*} w_i \del_{\beta}Z^{I^*} w_j - 
\frac{1}{2}\big(\del_t G_i^{j\alpha\beta}\big)\del_{\alpha}Z^{I^*} w_i \del_{\beta}Z^{I^*} w_i\bigg)dx
\leq M_i(s) \sum_i E_m(s,Z^{I^*} w_i)
$$

And By \eqref{quasilinear proof derived eq 2}, for any $|I|\leq 5$,
$$
\aligned
\quad\bigg(\sum_i E_m(s,Z^I w_i)\bigg)^{1/2}
&\leq \sqrt{3}\bigg(\sum_i E_m(B+1,Z^I w_i)\bigg)^{1/2}\exp\bigg(\int_{B+1}^s\sum_i3M_i(\tau)d\tau\bigg)
\\
&\quad+\int_{B+1}^s3\sum_iL_i(\tau)\exp\bigg(\int_{B+1}^s\sum_i3M_i(\tau')d\tau'\bigg)d\tau,
\endaligned
$$
where
$$
\bigg(\int_{H_s}\big|Z^IF_i - [Z^I, G_i^{j\alpha\beta}\del_{\alpha\beta}]w_j \big|^2dx\bigg) \leq L_i(s)
$$
and
$$
\int_{H_s}\frac{s}{t}\bigg(\big(\del_{\alpha}G_i^{j\alpha\beta}\big)\del_t Z^I w_i \del_{\beta}Z^I w_j - 
\frac{1}{2}\big(\del_t G_i^{j\alpha\beta}\big)\del_{\alpha}Z^I w_i \del_{\beta}Z^I w_i\bigg)dx
\leq M_i(s) \sum_i E_m(s,Z^I w_i)
$$

Suppose the following estimates can be deduced from \eqref{quasilinear decay on coefficients}, \eqref{quasilinear null conditions},
 \eqref{quasilinear proof energy assumption}, \eqref{quasilinear proof basic decay} and \eqref{quasilinear proof used decay}:
For any $|I|\leq 5$,
\begin{equation}\label{quasilinear proof source u}
\aligned
\bigg(\sum_{\jh}\int_{H_{\tau}}\big|Z^IF_{\ih} - Z^I\big(G_{\ih}^{j\alpha\beta}w_j\big)\big|^2 dx\bigg)^{1/2} 
&\leq C(n)K(C_1\varepsilon)^2\tau^{-1-\theta}
\\
&= L_{\jh}(s)
\endaligned
\end{equation}
with 
$$
\int_{B+1}^{\infty} L_{\jh} ds = C(n)K(C_1\varepsilon)^2\theta^{-1}(B+1)^{-\theta} = L < \infty.
$$
And for any $|I^*|\leq 7$ (if $|I^*|\leq 5$ then $N_i=0$):
\begin{equation}\label{quasilinear proof source}
\aligned
&\quad\bigg(\int_{H_{\tau}}\big|[G_i^{j\alpha\beta}\del_{\alpha\beta},Z^{I^*}]w_j + Z^{I^*}F_i(w,w')\big|^2 dx\bigg)^{1/2}
\\
&\leq C(n)K(C_1\varepsilon)^2 \tau^{-1} + C(n)KC_1\varepsilon \tau^{-1}\sum_{j\atop6\leq|I|\leq |I^*|}E_m(\tau, Z^{I^*} w_j)^{1/2}
\\
&=L_i(s) + N_i(s)\sum_{j\atop 6\leq|I|\leq |I^*|}E_m(\tau, Z^{I^*} w_i)^{1/2},
\endaligned
\end{equation}
\begin{equation}\label{quasilinear proof energy curveterm}
M_i(\tau) = C(n)KC_1\varepsilon \tau^{-1}.
\end{equation}
Then, for $|I|\leq 6$
$$
\aligned
E_m(s,Z^I u_{\jh})^{1/2} \leq E_m(s,Z^I u_{\jh})^{1/2} + C(n)K(C_1\varepsilon)^2\theta^{-1}(B+1)^{-\theta}.
\endaligned
$$
$$
\bigg(\sum_i E_m(s,Z^Iw_i)\bigg)^{1/2} \leq (\sqrt{3}C_0 + C_1)(s/B+1)^{C(j_0,k_0)KC_1\varepsilon}.
$$

Similarly, for $6\leq|I^*|\leq7$,
$$
\bigg(\sum_i E_m(s,Z^{I^*}w_i)\bigg)^{1/2} \leq (\sqrt{3}C_0 + C_1)(s/B+1)^{C(j_0,k_0)KC_1\varepsilon}.
$$
Now consider 
$$
T^*(\varepsilon):= \sup_{T}\big(\text{for any }B+1\leq s\leq T, \text{\eqref{quasilinear proof energy assumption} holds}\big).
$$
By continuity, at least one of the following two equations holds:
$$
\aligned
&\bigg(\sum_{i\atop |I^*|\leq 7}E_m(s,Z^I w_i)\bigg)^{1/2} = C_1\varepsilon s^{\delta},
\\
&E_m(s,Z^I u_{\ih})^{1/2} = C_1\varepsilon,\quad{for}\quad 1\leq \ih\leq j_0.
\endaligned
$$

When $KC(j_0,k_0)C_1\varepsilon = \delta,$ $C_1\geq 2C_0$ and $(B+1)^{\delta}> 2$
$$
\bigg(\sum_{i\atop |I^*|\leq 9}E_m(s,Z^I w_i)\bigg)^{1/2} < C_1\varepsilon s^{\delta}.
$$
When $KC(j_0,k_0)C_1\varepsilon = \delta,$ $C_1\geq 2C_0$ and $(B+1)^{\theta} > 2\delta^2\theta^{-1}(KC(j_0,k_0))^{-1}$,
$$
E_m(s,Z^I u_{\ih})^{1/2} < C_1\varepsilon,\quad{for}\quad 1\leq \ih\leq j_0.
$$
So for $\varepsilon$ sufficiently small, $C_1$ and $B$ sufficiently large, on time interval $[B+1,\infty)$, \eqref{quasilinear proof energy assumption} holds.
Then lemma \ref{basic lem prcp of cont.} completes the proof. 
\end{proof}
The remained work is to verify \eqref{quasilinear curved energy is big} and \eqref{quasilinear proof source u} - \eqref{quasilinear proof energy curveterm}
under the assumption of \eqref{quasilinear decay on coefficients},\eqref{quasilinear null conditions},\eqref{quasilinear proof energy assumption} 
and \eqref{quasilinear proof used decay}.

The following lemma is to guarantee \eqref{quasilinear curved energy is big}.
\begin{lemma}\label{quasilinear lem curved energy is big}
Suppose \eqref{quasilinear decay on coefficients} and \eqref{quasilinear proof basic decay} hold. Then following estimate holds
$$
\sum_iE_g(s,Z^I w_i) \leq 3\sum_iE_m(s,Z^I w_i).
$$
\end{lemma}
\begin{proof}
One notice that 
$$
\sum_{i,j,\alpha,\beta} \big|G_i^{j\alpha\beta}\big| \leq CK\sum_i\big(|\del w_i| + |w_i|\big).
$$
Then by simple calculation
$$
\aligned
\sum_i\big|E_G(s,w_i) - E_m(s,w_i) \big| 
&=  \big|2\int_{H_s} \big(\del_t w_i \del_{\beta}w_j G_i^{j\alpha\beta}\big) \cdot (1,-x^a/t) dx 
- \int_{H_s} \big(\del_{\alpha}w_i\del_{\beta}w_j G_i^{j\alpha\beta}\big)dx\big|
\\
&\leq 2\int_{H_s}\bigg(\sum_{i,j,\alpha,\beta}\big|G_i^{j\alpha\beta}\big|\bigg)\cdot\bigg(\sum_{\alpha,k}|\del_{\alpha}w_k|^2\bigg) dx
\\
&\leq 2CK\int_{H_s}\sum_i\big(|\del w_i| + |w_i|\big)\cdot\bigg(\sum_{\alpha,k}|\del_{\alpha}w_k|^2\bigg)dx
\\
&\leq 2CKC(n)C_1\varepsilon \int_{H_s}\big(t^{-3/2}s^{\delta} + t^{-1/2}s^{-1}\big)(t/s)^2\cdot\bigg(\sum_{\alpha,k}|(s/t)\del_{\alpha}w_k|^2\bigg)dx
\\
&=    2CKC(n)C_1\varepsilon \int_{H_s}\big(t^{1/2}s^{-2+\delta} + t^{3/2}s^{-3}\big)\cdot\bigg(\sum_{\alpha,k}|(s/t)\del_{\alpha}w_k|^2\bigg)dx
\\
&\leq 2C\delta \sum_iE_m(s,w_i).
\endaligned
$$
Here one takes $KC_1C(n)\varepsilon \leq \delta$. When $0<\delta$ small enough, 
$$
\sum_i\big|E_G(s,w_i) - E_m(s,w_i) \big| \leq 2\sum_iE_m(s,w_i),
$$
which complete the proof.
\end{proof}
All of the following lemmas are $L^2$ type estimates and their proofs are similar. The only simple idea used is to subtitle the decay 
estimates \eqref{quasilinear proof used decay} and the energy assumption \eqref{quasilinear proof energy assumption} into the expression. 
The calculation is long and tedious and will be left in Appendix B.

The first lemma is to guarantee \eqref{quasilinear proof source u}.
\begin{lemma}\label{quasilinear lem source u}
Suppose that \eqref{quasilinear decay on coefficients}, \eqref{quasilinear null conditions} and \eqref{quasilinear proof used decay} hold,
Then for any $|I|\leq 5$,
\begin{equation}\label{quasilinear lem esti-source1 u}
\bigg(\int_{H_s}\big|Z^I F_{\ih}(w,w')\big|^2dx\bigg)^{1/2} \leq C(n)(C_1\varepsilon)^2 K s^{-3/2+2\delta},
\end{equation}
and
\begin{equation}\label{quasilinear lem esti-source2 u}
\bigg(\int_{H_s}\big|Z^I\big(G_{\ih}^{j\alpha\beta}\del_{\alpha\beta}w_j\big)\big|^2dx\bigg)^{1/2} \leq C(n)(C_1\varepsilon)^2 K s^{-3/2 + 2\delta}.
\end{equation}
\end{lemma}
The last lemma is to guarantee \eqref{quasilinear proof source}
\begin{lemma}\label{quasilinear lem source-commutator}
Suppose \eqref{quasilinear decay on coefficients}, \eqref{quasilinear null conditions}, \eqref{quasilinear proof energy assumption} 
and \eqref{quasilinear proof used decay} hold,
then the following estimates hold for any $|I^*|\leq 7$:
\begin{equation}\label{quasilinear lem esti-source}
\bigg(\int_{H_s}\big|Z^{I^*} F_i(w,w')\big|^2dx\bigg)^{1/2} 
\leq C(n)(C_1\varepsilon)^2 K s^{-1} + C(n)C_1\varepsilon K s^{-1}\sum_{j\atop 6\leq|I|\leq |I^*|}E_m(s,Z^I w_j),
\end{equation}
\begin{equation}\label{quasilinear lem esti-commutator}
\bigg(\int_{H_s}\big|[Z^{I^*},G_i^{j\alpha\beta}\del_{\alpha\beta}]w_j \big|^2dx\bigg)^{1/2} 
\leq C(n)(C_1\varepsilon)^2 K s^{-1} + C(n)C_1\varepsilon K s^{-1}\sum_{j\atop 6\leq|I|\leq |I^*|} E_m(s,Z^I w_j),
\end{equation}
When $|I^*|\leq 5$ the last terms disappear.
\end{lemma}
The following lemma is to guarantee \eqref{quasilinear proof energy curveterm}:
\begin{lemma}
\label{quasilinear lem energy curveterm is small}
Suppose \eqref{quasilinear decay on coefficients},\eqref{quasilinear proof energy assumption} and \eqref{quasilinear proof used decay} hold,
then for any $|I^*|\leq 7$ the following estimates is true:
\begin{equation}\label{quasilinear lem esti M}
\int_{H_s}\frac{s}{t}\bigg(\big(\del_{\alpha}G_i^{j\alpha\beta}\big)\del_t Z^{I^*} w_i \del_{\beta}Z^{I^*} w_j 
- \frac{1}{2}\big(\del_t G_i^{j\alpha\beta}\big)\del_{\alpha}Z^{I^*} w_i \del_{\beta}Z^{I^*} w_j\bigg)dx
\leq M_i(s) \sum_kE_m(s,Z^I w_k),
\end{equation}
where
$$
M_i(s) = C(n)C_1\varepsilon K s^{-1}.
$$
\end{lemma}
\begin{proof}[Proof of lemma \ref{quasilinear lem source u}]
One will firstly prove \eqref{quasilinear lem esti-source1 u}. By \eqref{quasilinear decay on coefficients}, for any $|I|\leq 5$,
$$
F_{\ih} = P_{\ih}^{\alpha\beta jk}\del_{\alpha}w_j\del_{\beta}w_k + Q_{\ih}^{\alpha j\kc}\del_{\alpha}w_j v_{\kc} + R_{\ih}^{\jc\kc}v_{\jc}v_{\kc}.
$$
For the first term:
$$
\aligned
P_{\ih}^{\alpha\beta jk}\del_{\alpha}w_j\del_{\beta}w_k 
&= P_{\ih}^{\alpha\beta \jh\kh}\del_{\alpha}u_{\jh}\del_{\beta}u_{\kh} +P_{\ih}^{\alpha\beta \jh\kc}\del_{\alpha}u_{\jh}\del_{\beta}v_{\kc}
\\
&\quad P_{\ih}^{\alpha\beta \jc\kh}\del_{\alpha}v_{\jc}\del_{\beta}u_{\kh} +P_{\ih}^{\alpha\beta \jc\kc}\del_{\alpha}v_{\jc}\del_{\beta}v_{\kc}
\\
&=: Y_1 + Y_2 + Y_3 + Y_4
\endaligned
$$
Now consider $Y_1$. 
$$
\aligned
\bigg(\int_{H_s}\big|Z^I Y_1\big|^2dx\bigg)^{1/2} 
&\leq \bigg(\int_{H_s\cap\{r\leq t/2\}}\big|Z^I Y_1\big|^2dx\bigg)^{1/2} + \bigg(\int_{H_s\cap\{r\geq t/2\}}\big|Z^I Y_1\big|^2dx\bigg)^{1/2}
\\
&=: S_1 + S_2.
\endaligned
$$
$$
\aligned
S_1
&\leq \bigg(\int_{H_s\cap\{r\leq t/2\}}\big|Z^I \big(P_{\ih}^{\alpha\beta \jh\kh}\del_{\alpha}u_{\jh}\del_{\beta}u_{\kh}\big)\big|^2dx\bigg)^{1/2}
\\
&\leq \sum_{I_1 + I_2 =I}\bigg(\int_{H_s\cap\{r\leq t/2\}}\big|P_{\ih}^{\alpha\beta \jh\kh}Z^{I_1}\del_{\alpha}u_{\jh}Z^{I_2}\del_{\beta}u_{\kh}\big|^2dx\bigg)^{1/2}
\\
&\leq \sum_{|I_1|\leq 2\atop I_1 + I_2 =I}\bigg(\int_{H_s\cap\{r\leq t/2\}}\big|P_{\ih}^{\alpha\beta \jh\kh}Z^{I_1}\del_{\alpha}u_{\jh}Z^{I_2}\del_{\beta}u_{\kh}\big|^2dx\bigg)^{1/2}
\\
&\quad +\sum_{|I_2|\leq 2\atop I_1 + I_2 =I}\bigg(\int_{H_s\cap\{r\leq t/2\}}\big|P_{\ih}^{\alpha\beta \jh\kh}Z^{I_1}\del_{\alpha}u_{\jh}Z^{I_2}\del_{\beta}u_{\kh}\big|^2dx\bigg)^{1/2}
\\
&\leq \sum_{|I_1|\leq 2\atop I_1 + I_2 =I}\bigg(\int_{H_s\cap\{r\leq t/2\}}\big|KC(n)C_1\varepsilon t^{-1/2}s^{-1}(t/s))\big|^2\cdot\big|(s/t)Z^{I_2}\del_{\beta}u_{\kh}\big|^2dx\bigg)^{1/2}
\\
&\quad +\sum_{|I_2|\leq 2\atop I_1 + I_2 =I}\bigg(\int_{H_s\cap\{r\leq t/2\}}\big|(s/t)Z^{I_1}\del_{\alpha}u_{\jh}\big|^2\cdot\big|KC(n)C_1\varepsilon t^{-1/2}s^{-1}(t/s)\big|^2dx\bigg)^{1/2}
\\
&\leq C(n)(C_1\varepsilon)^2Ks^{-3/2}.
\endaligned
$$
To estimate $S_2$, one uses the ``one-one" frame and the weak null conditions \eqref{quasilinear null conditions}.
\begin{equation}\label{quasilinear lem example of null1}
\aligned
S_2
& = \bigg(\int_{H_s\cap\{r\geq t/2\}}\big|Z^I\big(\Pu_{\ih}^{\alpha\beta \jh\kh}\delu_{\alpha}u_{\jh}\delu_{\beta}u_{\kh}\big)\big|^2dx\bigg)^{1/2}
\\
&\leq \bigg(\int_{H_s\cap\{r\geq t/2\}}\big|Z^I\big(\Pu_{\ih}^{a\beta \jh\kh}\delu_au_{\jh}\delu_{\beta}u_{\kh}\big)\big|^2dx\bigg)^{1/2}
\\
&\quad + \bigg(\int_{H_s\cap\{r\geq t/2\}}\big|Z^I\big(\Pu_{\ih}^{\alpha b\jh\kh}\delu_{\alpha}u_{\jh}\delu_b u_{\kh}\big)\big|^2dx\bigg)^{1/2}
\\
& =: S_2^{(1)} + S_2^{(2)}.
\endaligned
\end{equation}
By lemma \ref{basic lem frame},
\begin{equation}\label{quasilinear lem example of null2}
\aligned
S_2^{(1)}
& \leq \sum_{I_1+I_2+I_3=I}\bigg(\int_{H_s\cap\{r\geq t/2\}}\big|Z^{I_3}\Pu_{\ih}^{a\beta\jh\kh}Z^{I_1}\delu_au_{\ih}Z^{I_2}\delu_{\beta}u_{\kh}\big|^2dx\bigg)^{1/2}
\\
& \leq \sum_{|I_1\leq 2\atop I_1+I_2+I_3=I}\bigg(\int_{H_s\cap\{r\geq t/2\}}\big|Z^{I_3}\Pu_{\ih}^{a\beta\jh\kh}Z^{I_1}\delu_au_{\ih}Z^{I_2}\delu_{\beta}u_{\kh}\big|^2dx\bigg)^{1/2}
\\
&\quad+ \sum_{|I_2|\leq 2\atop I_1+I_2+I_3=I}\bigg(\int_{H_s\cap\{r\geq t/2\}}\big|Z^{I_3}\Pu_{\ih}^{a\beta\jh\kh}Z^{I_1}\delu_au_{\ih}Z^{I_2}\delu_{\beta}u_{\kh}\big|^2dx\bigg)^{1/2}
\\
&\leq C(n)(C_1\varepsilon)^2 K s^{-3/2}.
\endaligned
\end{equation}
$S_2^{(2)}$ is estimated in the same way. Then
$$
\bigg(\int_{H_s}\big|Z^IY_1\big|^2\bigg)^{1/2} \leq C(n)(C_1\varepsilon)^2s^{-3/2}.
$$

Now consider the term of $Y_2,$ and $Y_3$. One notices that by \eqref{quasilinear proof energy assumption},
for any $|I|\leq 6$ and $|J|\leq 5$,
\begin{equation}\label{quasilinear energy better v}
\aligned
&\bigg(\int_{H_s}\big|Z^I\del_{\alpha}v_{\kc}\big|^2dx\bigg)^{1/2}\leq C_1\varepsilon,
\\
&\bigg(\int_{H_s}\big|Z^J\del_{\alpha\beta}v_{\kc}\big|^2dx\bigg)^{1/2}\leq C_1\varepsilon.
\endaligned
\end{equation}
Taking this into account, one simply substitutes \eqref{quasilinear proof energy assumption} and \eqref{quasilinear proof used decay} into
the expression of $Y_2$ and $Y_3$ (null conditions are not imposed here) and gets
$$
\bigg(\int_{H_s}\big|Z^IY_2\big|^2dx\bigg)^{1/2} + \bigg(\int_{H_s}\big|Z^IY_2\big|^2dx\bigg)^{1/2} \leq C(n)(C_1\varepsilon)^{1/2}s^{-3/2+\delta}
$$
The estimate on $Y_4$ is even simpler than that of $Y_2$ and $Y_3$. One simply substitutes \eqref{quasilinear proof energy assumption} and
\eqref{quasilinear proof used decay} into the expression, and gets
$$
\bigg(\int_{H_s}\big|Z^IY_4\big|^2dx\bigg)^{1/2} \leq C(n)(C_1\varepsilon)^2s^{-3/2+2\delta}.
$$

The estimate of the integrals 
$$
\int_{H_s}\big|Z^I\big(Q_{\ih}^{\alpha j\kc}\del_{\alpha}w_j v_{\kc}\big)\big|^2dx
$$
and
$$
\int_{H_s}\big|Z^I\big(R_{\ih}^{\jc\kc}v_{\jc}v_{\kc}\big)\big|^2dx,
$$
are just substitutions of \eqref{quasilinear proof energy assumption} and \eqref{quasilinear proof used decay}. One gets
$$
\int_{H_s}\big|Z^I\big(Q_{\ih}^{\alpha j\kc}\del_{\alpha}w_j v_{\kc}\big)\big|^2dx
 \leq C(n)(C_1\varepsilon)^2 K s^{-3/2 + 2\delta},
$$
and 
$$
\int_{H_s}\big|Z^I\big(R_{\ih}^{\jc\kc}v_{\jc}v_{\kc}\big)\big|^2dx
 \leq C(n)(C_1\varepsilon)^2 K s^{-3/2 + 2\delta}.
$$
So one gets for any $|I|\leq 8$,
$$
\bigg(\int_{H_s}\big|Z^I F_{\ih}\big|^2\bigg)^{1/2} \leq C(n)(C_1\varepsilon)^2 s^{-3/2 + 2\delta}.
$$
The proof of \eqref{quasilinear lem esti-source2 u} is quite similar. 
$$
Z^IG_{\ih}^{j\alpha\beta}\del_{\alpha\beta}w_j = Z^IG_{\ih}^{\jh\alpha\beta}\del_{\alpha\beta}u_{\jh} + Z^IG_{\ih}^{\jc\alpha\beta}\del_{\alpha\beta}v_{\jc}.
$$
The first term is decomposed into three pieces:
$$
\aligned
Z^IG_{\ih}^{\jh\alpha\beta}\del_{\alpha\beta}u_{\jh}
&   = Z^I\big(A_{\ih}^{\jh\alpha\beta\gamma\kh}\del_{\gamma}u_{\kh}\del_{\alpha\beta}u_{\jh}\big)
    + Z^I\big(A_{\ih}^{\jh\alpha\beta\gamma\kc}\del_{\gamma}v_{\kc}\del_{\alpha\beta}u_{\jh}\big)
    + Z^I\big(B_{\ih}^{\jh\alpha\beta\kc}v_{\kc}\del_{\alpha\beta}u_{\jh}\big)
\\
&=: M_1(s) + M_2(s) + M_3(s).
\endaligned
$$
The estimate on $M_2$ is simple.
$$
\aligned
M_2(s) 
&\leq \bigg(\int_{H_s}\big|Z^I\big(A_{\ih}^{\jh\alpha\beta\gamma\kc}\del_{\gamma}v_{\kc}\del_{\alpha\beta}u_{\jh}\big)\big|^2 dx\bigg)^{1/2}
\\
&\leq \sum_{I_1+I_2=I}\bigg(\int_{H_s}K^2\big|Z^{I_1}\del_{\gamma}v_{\kc}\big|^2\cdot\big|Z^{I_2}\del_{\alpha\beta}u_{\jh}\big|^2dx\bigg)^{1/2}
\\
&\leq \sum_{|I_1|\leq 2 \atop I_1+I_2=I}\bigg(\int_{H_s}K^2\big|Z^{I_1}\del_{\gamma}v_{\kc}\big|^2\cdot\big|Z^{I_2}\del_{\alpha\beta}u_{\jh}\big|^2dx\bigg)^{1/2}
\\
&\quad+\sum_{|I_2|\leq 2 \atop I_1+I_2=I}\bigg(\int_{H_s}K^2\big|Z^{I_1}\del_{\gamma}v_{\kc}\big|^2\cdot\big|Z^{I_2}\del_{\alpha\beta}u_{\jh}\big|^2dx\bigg)^{1/2}
\\
& \leq KC(n)C_1\varepsilon \sum_{|I_2|\leq 5}\bigg(\int_{H_s}\big(t^{-3/2}s^{\delta}(t/s)\big)^2\cdot\big|(s/t)Z^{I_2}\del_{\alpha\beta}u_{\jh}\big|^2dx\bigg)^{1/2}
\\
&\quad + KC(n)C_1\varepsilon \sum_{|I_1|\leq 5}\bigg(\int_{H_s}\big|Z^{I_1}\del_{\gamma}v_{\kc}\big|^2\big(t^{-1/2}s^{-1}\big)^2dx\bigg)^{1/2}
\\
& \leq KC(n)(C_1\varepsilon)^2 s^{-3/2+2\delta}.
\endaligned
$$
Similarly 
$$
M_3 \leq C(n)(C_1\varepsilon)^2K s^{-3/2 +2\delta}.
$$
The estimate of $M_1(s)$ is the more difficult than the others.
$$
\aligned
M_1(s) 
&\leq \bigg(\int_{H_s\cap\{r\leq t/2\}}\big|Z^I\big(A_{\ih}^{\jh\alpha\beta\gamma\kh}\del_{\gamma}u_{\kh}\del_{\alpha\beta}u_{\jh}\big)\big|^2\bigg)^{1/2}
\\
&\quad + \bigg(\int_{H_s\cap\{r\leq t/2\}}\big|Z^I\big(A_{\ih}^{\jh\alpha\beta\gamma\kh}\del_{\gamma}u_{\kh}\del_{\alpha\beta}u_{\jh}\big)\big|^2\bigg)^{1/2}
\\
& =: S_1 + S_2
\endaligned
$$
$$
\aligned
S_1 
&\leq 
\sum_{|I_1|\leq 2\atop I_1+I_2=I}\bigg(\int_{H_s\cap\{r\leq t/2\}}K^2\big|Z^{I_1}\del_{\gamma}u_{\kh}\big|^2\cdot\big|Z^{I_2}\del_{\alpha\beta}u_{\jh}\big)\big|^2\bigg)^{1/2}
\\
&\quad+\sum_{|I_2|\leq 2\atop I_1+I_2=I}\bigg(\int_{H_s\cap\{r\leq t/2\}}K^2\big|Z^{I_1}\del_{\gamma}u_{\kh}\big|^2\cdot\big|Z^{I_2}\del_{\alpha\beta}u_{\jh}\big|^2\bigg)^{1/2}
\\
&\leq C(n)(C_1\varepsilon)^2 s^{-3/2 + \delta}.
\endaligned
$$
$$
\aligned
S_2
&\leq \bigg(\int_{H_s\cap\{r\geq t/2\}}\big|Z^I\big(\Au_{\ih}^{\jh\alpha\beta\gamma\kh}\delu_{\gamma}u_{\kh}\delu_{\alpha\beta}u_{\jh}\big)\big|^2\bigg)^{1/2}
\\
&\quad 
-\bigg(\int_{H_s\cap\{r\geq t/2\}}\big|Z^I\big(\Au_{\ih}^{\jh\alpha\beta\gamma\kh}\delu_{\gamma}u_{\kh}\delu_{\alpha}\big(\Phi_{\beta}^{\beta'}\big)\del_{\beta'}u_{\jh}\big)\big|^2\bigg)^{1/2}
\\
& := H_1(s) + H_2(s)
\endaligned
$$
The estimate on $H_2(s)$ is simple. One notice that $\delu_{\alpha}\Phi_{\beta}^{\beta'}\leq Ct^{-1}$ when $r\geq t/2$. Then one gets
$$
H_2(S) \leq C(n)(C_1\varepsilon)^2K s^{-5/2}.
$$
The estimate of $H_1(s)$ will consult the weak null conditions \eqref{quasilinear null conditions}. Just as one as shown in \eqref{quasilinear lem example of null1}
and \eqref{quasilinear lem example of null2}, 
$$
H_1(s) \leq C(n)(C_1\varepsilon)^2K s^{-3/2+\delta}.
$$
To estimate the term 
$$
\int_{H_s}\big|Z^I\big(G_{\ih}^{\jc\alpha\beta}\del_{\alpha\beta}v_{\jc}\big)\big|^2dx,
$$
One notices that 
\begin{equation}\label{quasilinear lem G-estimate}
\big|Z^IG_{\ih}^{\jc\alpha\beta}\big|
\leq C(n)K\sum_{\alpha\atop|I'|\leq |I|}\bigg(\sum_{\kh}Z^{I'}\del_{\alpha}u_{\kh} + \sum_{\lc} Z^{I'}\del_{\alpha}v_{\lc} + \sum_{\lc}Z^{I'}v_{\lc}\bigg).
\end{equation}
So one gets
$$
\aligned
&\quad\bigg(\int_{H_s}\big|Z^I\big(G_{\ih}^{\jc\alpha\beta}\del_{\alpha\beta}v_{\jc}\big)\big|^2dx\bigg)^{1/2}
\\
&\leq \sum_{|I_1|\leq 2}\bigg(\int_{H_s}\big|Z^{I_1}G_{\ih}^{\jc\alpha\beta}Z^{I_2}\del_{\alpha\beta}v_{\jc}\big|^2dx\bigg)^{1/2}
\\
&\quad + \sum_{|I_2|\leq 2}\bigg(\int_{H_s}\big|Z^{I_1}G_{\ih}^{\jc\alpha\beta}Z^{I_2}\del_{\alpha\beta}v_{\jc}\big|^2dx\bigg)^{1/2}
\\
&\leq C(n)C_1\varepsilon K \bigg(\int_{H_s}\big(t^{-1/2}s^{-1} + t^{-3/2}s^{\delta}\big)^2\cdot\big|Z^{I_2}\del_{\alpha\beta}v_{\jc}\big|^2dx\bigg)^{1/2}
\\
&\quad + C(n)C_1\varepsilon K \bigg(\int_{H_s}\big|Z^{I_1}G_{\ih}^{\jc\alpha\beta}\big|^2\cdot\big(t^{-3/2}s^{\delta}\big)^2dx\bigg)^{1/2}
\\
&\leq C(n)(C_1\varepsilon)^2 Ks^{-3/2+2\delta}.
\endaligned
$$
\end{proof}
\begin{proof}[Proof of lemma \ref{quasilinear lem source-commutator}]
One will firstly prove \eqref{quasilinear lem esti-source}. 
$$
F_i = P_i^{\alpha\beta jk}\del_{\alpha\beta}w_j \del_{\beta}w_k + Q_i^{\alpha j\kc}\del_{\alpha}w_j v_{\kc} + R_i^{\jc\kc}v_{\jc}v_{\kc}.
$$
And
$$
\aligned
\quad P_i^{\alpha\beta jk}\del_{\alpha\beta}w_j \del_{\beta}w_k
&=P_i^{\alpha\beta \jh\kh}\del_{\alpha\beta}u_{\jh} \del_{\beta}u_{\kh} + P_i^{\alpha\beta \jh\kc}\del_{\alpha\beta}u_{\jh} \del_{\beta}v_{\kc}
\\
&\quad + P_i^{\alpha\beta \jc\kh}\del_{\alpha\beta}v_{\jc} \del_{\beta}u_{\kh} + P_i^{\alpha\beta \jc\kc}\del_{\alpha\beta}v_{\jc} \del_{\beta}v_{\kc}.
\endaligned
$$ 
Then the following estimates hold. Here $|I^*|\leq 8$. For the first term:
$$
\aligned
&\quad\bigg(\int_{H_s}\big|Z^{I^*}\big(P_i^{\alpha\beta\jh\kh}\del_{\alpha}u_{\jh}\del_{\beta}u_{\kh}\big)\big|^2dx\bigg)^{1/2}
\\
&\leq K\sum_{I^*_1+I^*_2=I^*\atop |I^*_1|\leq 3\ |I^*_2|\leq 5}
\bigg(\int_{H_s}\big|Z^{I_1^*}\del_{\alpha}u_{\jh}\big|^2\cdot\big|Z^{I^*_2}\del_{\beta}u_{\kh}\big|^2dx\bigg)^{1/2}
\\
&\quad+K\sum_{I^*_1+I^*_2=I^*\atop |I^*_2|\leq 3 |I^*_1|\leq 5}
\bigg(\int_{H_s}\big|Z^{I_1^*}\del_{\alpha}u_{\jh}\big|^2\cdot\big|Z^{I^*_2}\del_{\beta}u_{\kh}\big|^2dx\bigg)^{1/2}
\\
&+\quad K\sum_{I^*_1+I^*_2=I^*\atop |I^*_2|\geq 6}
\bigg(\int_{H_s}\big|Z^{I_1^*}\del_{\alpha}u_{\jh}\big|^2\cdot\big|Z^{I^*_2}\del_{\beta}u_{\kh}\big|^2dx\bigg)^{1/2}
\\
&\quad+K\sum_{I^*_1+I^*_2=I^*\atop |I^*_1|\geq 6}
\bigg(\int_{H_s}\big|Z^{I_1^*}\del_{\alpha}u_{\jh}\big|^2\cdot\big|Z^{I^*_2}\del_{\beta}u_{\kh}\big|^2dx\bigg)^{1/2}
\\
&\leq C(n)(C_1\varepsilon)^2 s^{-1} + C(n)C_1\varepsilon Ks^{-1} \sum_{\ih\atop 6\leq|I'|\leq7}E_m(s,Z^{I'}u_{\ih})^{1/2}.
\endaligned
$$
For the second term:
$$
\aligned
&\quad\bigg(\int_{H_s}\big|Z^{I^*}\big(P_i^{\alpha\beta\jh\kc}\del_{\alpha}u_{\jh}\del_{\beta}v_{\kc}\big)\big|^2dx\bigg)^{1/2}
\\
&\leq K\sum_{I^*_1+I^*_2=I^*\atop|I^*_1|\leq 3 |I^*_2|\leq 6}
\bigg(\int_{H_s}\big|Z^{I^*_1}\del_{\alpha}u_{\jh}\big|^2\cdot\big|Z^{I^*_2}\del_{\beta}v_{\kc}\big|^2dx\bigg)^{1/2}
 +K\bigg(\int_{H_s}\big|\del_{\alpha}u_{\jh}\big|^2\cdot\big|Z^{I^*}\del_{\beta}v_{\kc}\big|^2dx\bigg)^{1/2}
\\
&\quad +K\sum_{|I^*_2|\leq 3\atop I^*_1+I^*_2=I^*}
\bigg(\int_{H_s}\big|Z^{I^*_1}\del_{\alpha}u_{\jh}\big|^2\cdot\big|Z^{I^*_2}\del_{\beta}v_{\kc}\big|^2dx\bigg)^{1/2}
\\
&\leq KC(n)(C_1\varepsilon)^2s^{-3/2 + \delta} + KC(n)C_1\varepsilon s^{-1} \sum_{\kc}E_m(s,Z^{I^*}v_{\kc})^{1/2} 
+ KC(n)(C_1\varepsilon)^2s^{-3/2+2\delta}.
\endaligned
$$
The estimate on the third term is the same.
For the forth term, 
$$
\aligned
&\quad\bigg(\int_{H_s}\big|Z^{I^*}\big(P_i^{\alpha\beta\jc\kc}\del_{\alpha}v_{\jc}\del_{\beta}v_{\kc}\big|^2dx\bigg)^{1/2}
\\
&\leq K\sum_{|I^*_1|\leq 3\atop I^*_1+I^*_2=I^*}
\bigg(\int_{H_s}\big|Z^{I^*_1}\del_{\alpha}v_{\jc}\big|^2\cdot\big|Z^{I^*_2}\del_{\beta}v_{\kc}\big|^2dx\bigg)^{1/2}
     +\sum_{|I^*_2|\leq 3\atop I^*_1+I^*_2=I^*}
\bigg(\int_{H_s}\big|Z^{I^*_1}\del_{\alpha}v_{\jc}\big|^2\cdot\big|Z^{I^*_2}\del_{\beta}v_{\kc}\big|^2dx\bigg)^{1/2}
\\
&\leq KC(n)(C_1\varepsilon)^{-3/2+2\delta}.
\endaligned
$$
The estimates on terms about $Q_i^{\alpha j\kc}$ and $R_i^{\jc\kc}$ are similar. One omits the details.  

Now one will prove \eqref{quasilinear lem esti-commutator}. In general one has he following decomposition:
$$
[G_i^{j\alpha\beta}\del_{\alpha\beta},Z^{I^*}] 
= \sum_{|I^*_1|\geq 1\atop I^*_1 + I^*_2 = I^*}Z^{I^*_1}G_i^{j\alpha\beta}Z^{I^*_2}\del_{\alpha\beta}w_j
+ G_i^{j\alpha\beta}[Z^{I^*},\del_{\alpha\beta}]w_j.
$$
The estimate no the second term is simple. One notices that, by \eqref{commutator H-partial}, $[Z^{I^*},\del_{\alpha\beta}]w_j$ is finite linear combination of
$\del^Jw_j$ with $|J|\leq 7$. So one has
$$
\aligned
&\quad\bigg(\int_{H_s}\big|G_i^{j\alpha\beta}[Z^{I^*},\del_{\alpha\beta}]w_j\big|^2dx\bigg)^{1/2}
\\
&\leq\sum_{|J|\leq 6}\sum_{\ih}\bigg(\int_{H_s}\big|(t/s)G_i^{\jh\alpha\beta}\big|^2\cdot\big|(s/t)Z^Ju_{\ih}\big|^2dx\bigg)^{1/2} 
    +\sum_{|J|\leq 6}\sum_{\jc}\bigg(\int_{H_s}\big|G_i^{\jc\alpha\beta}\big|^2\cdot\big|Z^Jv_{\jc}\big|^2dx\bigg)^{1/2}
\\
&\quad +\sum_{|J|\geq 6}\sum_{\ih}\bigg(\int_{H_s}\big|(t/s)G_i^{\jh\alpha\beta}\big|^2\cdot\big|(s/t)Z^Ju_{\ih}\big|^2dx\bigg)^{1/2}
\\
&\quad +\sum_{|J|  =  7}\sum_{\jc}\bigg(\int_{H_s}\big|(t/s)G_i^{\jc\alpha\beta}\big|^2\cdot\big|(s/t)Z^Jv_{\jc}\big|^2dx\bigg)^{1/2}.
\\
\endaligned
$$
Then by \eqref{quasilinear lem G-estimate}, one gets
$$
\aligned
&\quad\bigg(\int_{H_s}\big|G_i^{j\alpha\beta}[Z^{I^*},\del_{\alpha\beta}]w_j\big|^2dx\bigg)^{1/2}
\\
&\leq C(n)C_1\varepsilon \sum_{|J|\leq 5}\sum_{\ih}
\bigg(\int_{H_s}\big|(t/s)(t^{-1/2}s^{-1} + t^{-3/2}s^{\delta})\big|^2\cdot\big|(s/t)Z^Ju_{\ih}\big|^2dx\bigg)^{1/2}
\\
&\quad+C(n)C_1\varepsilon \sum_{|J|\leq 6}\sum_{\jc}
\bigg(\int_{H_s}\big|t^{-1/2}s^{-1} + t^{-3/2}s^{\delta}\big|^2\cdot\big|Z^Jv_{\jc}\big|^2dx\bigg)^{1/2}
\\
&+\sum_{|J|\geq 6}\sum_{\ih}
\bigg(\int_{H_s}\big|(t/s)(t^{-1/2}s^{-1} + t^{-3/2}s^{\delta})\big|^2\cdot\big|(s/t)Z^Ju_{\ih}\big|^2dx\bigg)^{1/2}
\\
&\quad +\sum_{|J|  =  7}\sum_{\jc}
\bigg(\int_{H_s}\big|(t/s)(t^{-1/2}s^{-1} + t^{-3/2}s^{\delta})\big|^2\cdot\big|(s/t)Z^Jv_{\jc}\big|^2dx\bigg)^{1/2}
\\
&\leq C(n)(C_1\varepsilon)^2Ks^{-1} + C(n)(C_1\varepsilon)Ks^{-3/2 + \delta} + C(n)(C_1\varepsilon)Ks^{-3/2 + 2\delta} 
\\
&\quad + C(n)C_1\varepsilon Ks^{-1}\sum_{|J|\geq 6}\sum_{\ih}E_m(s,Z^J u_{\ih})^{1/2} 
+ C(n)C_1\varepsilon Ks^{-1}\sum_{|J|=7}\sum_{\jc}E_m(s,Z^J v_{\jc})^{1/2}
\\
&\leq C(n)(C_1\varepsilon)Ks^{-3/2 + 2\delta} 
+ C(n)C_1\varepsilon Ks^{-1}\bigg(\sum_{|J|\geq 6}\sum_{\ih}E_m(s,Z^I u_{\ih})^{1/2} + \sum_{|J|=7}\sum_{\jc}E_m(s,Z^I v_{\jc})^{1/2}\bigg)
\endaligned
$$
The estimate on the terms about 
$$
\sum_{|I^*_1|\geq 1\atop I^*_1 + I^*_2 = I^*}Z^{I^*_1}G_i^{j\alpha\beta}Z^{I^*_2}\del_{\alpha\beta}w_j
$$
is similar to that of \eqref{quasilinear lem esti-source}. With the aid of \eqref{quasilinear lem G-estimate}:

$$
\aligned
&\quad\sum_{|I^*_1|\geq 1\atop I^*_1 + I^*_2 = I^*}\bigg(\int_{H_s}\big|Z^{I^*_1}G_i^{j\alpha\beta}Z^{I^*_2}\del_{\alpha\beta}w_j\big|^2dx\bigg)^{1/2}
\\
&\leq \sum_{|I^*_2|= 6\atop I^*_1 + I^*_2 = I^*}
\bigg(\int_{H_s}\big|Z^{I^*_1}G_i^{j\alpha\beta}Z^{I^*_2}\del_{\alpha\beta}w_j\big|^2dx\bigg)^{1/2}
     +\sum_{|I^*_2|= 5\atop I^*_1 + I^*_2 = I^*}
\bigg(\int_{H_s}\big|Z^{I^*_1}G_i^{\jh\alpha\beta}Z^{I^*_2}\del_{\alpha\beta}u_{\jh}\big|^2dx\bigg)^{1/2}
\\
&\quad+\sum_{2\leq|I^*_1|\leq 3\atop I^*_1 + I^*_2 = I^*}
\bigg(\int_{H_s}\big|Z^{I^*_1}G_i^{\jc\alpha\beta}Z^{I^*_2}\del_{\alpha\beta}v_{\jc}\big|^2dx\bigg)^{1/2}
      +\sum_{|I^*_1|= 3\atop I^*_1 + I^*_2 = I^*}
\bigg(\int_{H_s}\big|Z^{I^*_1}G_i^{\jh\alpha\beta}Z^{I^*_2}\del_{\alpha\beta}u_{\jh}\big|^2dx\bigg)^{1/2}
\\
&\quad + \sum_{1\leq|I^*_2|\leq 3\atop I^*_1 + I^*_2 = I^*}
\bigg(\int_{H_s}\big|Z^{I^*_1}G_i^{j\alpha\beta}Z^{I^*_2}\del_{\alpha\beta}w_j\big|^2dx\bigg)^{1/2}
+
\bigg(\int_{H_s}\big|Z^{I^*}G_i^{j\alpha\beta}\del_{\alpha\beta}w_j\big|^2dx\bigg)^{1/2}
\endaligned
$$
Now take into account \eqref{quasilinear proof energy assumption}, \eqref{quasilinear proof used decay} and
\eqref{quasilinear energy better v},
$$
\aligned
&\quad\sum_{|I^*_1|\geq 1\atop I^*_1 + I^*_2 = I^*}\bigg(\int_{H_s}\big|Z^{I^*_1}G_i^{j\alpha\beta}Z^{I^*_2}\del_{\alpha\beta}w_j\big|^2dx\bigg)^{1/2}
\\
&\leq KC(n)C_1\varepsilon s^{-1}\sum_{i\atop |I|=7}E_m(s,Z^I w_i)^{1/2} 
+ \bigg(KC(n)C_1\varepsilon s^{-1}\sum_{\ih\atop |I|\geq6}E_m(s,Z^I u_{\ih})^{1/2} + KC(n)(C_1\varepsilon)^2 s^{-1}\bigg)
\\
&\quad + KC(n)(C_1\varepsilon)^2 s^{-3/2+2\delta} + KC(n)(C_1\varepsilon)^2s^{-1}
\\
&\quad + KC(n)(C_1\varepsilon)^2 s^{-1} + KC(n)C_1\varepsilon s^{-1}\sum_i\sum_{|I|\geq 6} E_m(s,Z^I w_i).
\endaligned
$$
So finally one concludes by \eqref{quasilinear lem esti-commutator}. 
\end{proof}
\begin{proof}[Proof of lemma \ref{quasilinear lem energy curveterm is small}]
One will firstly prove the following estimate:
$$
\int_{H_s}(s/t)\del_t Z^{I^*}w_i \del_{\beta}Z^{I^*}w_j \del_{\alpha}G_i^{j\alpha\beta}dx \leq C(n)C_1\varepsilon s^{-1}\sum_k E_m(s,Z^{I^*}w_k).
$$
By \eqref{quasilinear lem G-estimate}, 
$$
\big|\del_{\alpha}G_i^{j\alpha\beta}\big| \leq C(n)C_1\varepsilon K(t^{-1/2}s^{-1} + t^{-3/2}s^{\delta}).
$$
Substitute this into the expression, one gets:
$$
\aligned
&\quad\int_{H_s}(s/t)\del_t Z^{I^*}w_i \del_{\beta}Z^{I^*}w_j \del_{\alpha}G_i^{j\alpha\beta}dx
\\
&\leq C(n)C_1\varepsilon K\int_{H_s}\big(t^{-1/2}s^{-1} + t^{-3/2}s^{\delta}\big)(t/s)\cdot\big|(s/t)\del_t Z^{I^*}w_i\big|\cdot\big|(s/t)Z^{I^*}\del_{\beta}w_j\big|dx
\\
&\leq C(n)C_1\varepsilon K s^{-1}\sum_kE_m(s,Z^{I^*}w_k).
\endaligned
$$
\end{proof}
\appendix
\section{Local existence for small initial data}
One will establish the following local-in-time existence result for small initial data. The interest is to get an a priori estimate on
the life spin time.
Consider the Cauchy problem in $\mathbb{R}^{n+1}$:
\begin{equation}\label{appendix eq nonlinear}
\left\{
\aligned
&g_i^{\alpha\beta}(w,\del w)\del_{\alpha\beta} w_i + D_i^2 w_i = F_i(w,\del w),
\\
&w_i(B+1,x) =\varepsilon'{w_i}_0,\quad \del_t w_i(B+1,x) = \varepsilon'{w_i}_1.
\endaligned
\right.
\end{equation}
Here 
$$
\aligned
&g_i(w,\del w) = m^{\alpha\beta} A^{\alpha\beta\gamma j}_i\del_{\gamma}w_j + B^{\alpha\beta j}w_j + O(|w|^2 + |\del w|^2),
\\
&F_i(w,\del w) = P^{\alpha\beta jk}_i \del_{\alpha}w_j \del_{\beta}w_k + Q^{\alpha jk}_i \del_{\alpha}w_j w_k + R^{jk}_i w_j w_k + O(|w|^3 + |\del w|^3).
\endaligned
$$
These $A^{\alpha\beta\gamma j}_i,B^{\alpha\beta j},P^{\alpha\beta jk}_i,Q^{\alpha jk}_i,R^{jk}_i$ are constants.
$({w_i}_0,{w_i}_1)\in H^{s+1}\times H^s$ functions and supported on the disc $\{|x| \leq B\}$. In general the following local-in-time existence holds
\begin{theorem}\label{appendix Thm A}
For any integer $s\geq 2p(n)-1$,
there exists a time interval $[0,T(\varepsilon')]$ on which
the cauchy problem \eqref{appendix eq nonlinear} has an unique solution in sense of distribution $w_i(t,x)$. Further more 
$$
w_i(t,x) \in C([0,T(\varepsilon')], H^{s+1}) \cap C^1([0,T(\varepsilon')], H^s),
$$
and when $\varepsilon'$ sufficiently small, 
$$
T(\varepsilon') \geq C(A\varepsilon')^{-1/2}
$$
where $A$ is a constant depending only on ${w_i}_0$ and ${w_i}_1$.
Let $E_g(T,w_i)$ be the hyperbolic energy defined in the section 2.2. For any $\varepsilon, C_1>0$, there exists an $\varepsilon'$ such that
$$
\sum_{i}E_g(B+1, w_i) \leq C_1\varepsilon.
$$
\end{theorem} 
\begin{proof}
The proof is just a classical iteration procedure. The high-order terms will be omitted. One will not give the details but the key steps.
One defines the standard energy associated to a curved metric $g$
$$
E^*_g(s,w_i) := \int_{\mathbb{R}^n} \big(g^{00}(\del_t u)^2 - g^{ij}\del_iu\del_ju\big) dx.
$$
One takes the following iteration procedure:
\begin{equation}\label{appendix proof iteration}
\left\{
\aligned
&g^{\alpha\beta}_i(w^k,\del w^k) \del_{\alpha\beta} w_i^{k+1} = F(w^k,\del w^k),
\\
&w_i(0,x) = \varepsilon'{w_i}_0,\quad \del_t w_i(0,x) = \varepsilon'{w_i}_0,
\endaligned
\right.
\end{equation}
and take $w_i^0$ as the solution of the following linear Cauchy problem:
$$
\left\{
\aligned
& \Box w_i = 0,
\\
& w_i(0,x) = \varepsilon'{w_i}_0, \quad \del_t w_i (0,x) = \varepsilon'{w_i}_1.
\endaligned
\right.
$$
Suppose that for any $|I|\leq 2p(n)-1$,
\begin{equation}\label{appendix proof energy assumption}
\aligned
&\varepsilon'A \geq e\cdot E^*_g(B+1,\del^I w_i^k)^{1/2},
\\
&\varepsilon'A \geq E^*_g(t,\del^I w_i^k)^{1/2}.
\endaligned
\end{equation}
Taking the size of the support of the solution $w_i^k(t,\cdot)$ into consideration, by Sobolev's inequality, for any $|J| \leq p(n)-1$,
\begin{equation}\label{appendix proof decay estimate}
|\del^J w_i^k|(t,x) \leq C(t+B+1) \varepsilon' A.
\end{equation}
Now one wants to get the energy estimate on $\del^I w_i^{k+1}$. By the same method used in \cite{So}, one gets
$$
\aligned
E^*_g(t,\del^I w_i^{k+1})^{1/2} 
&\leq E_g(t,\del^I w_i^{k+1}) \exp\bigg(CA\varepsilon'\int_{B+1}^t(\tau+B+1)d\tau\bigg)
\\
&\leq e^{-1} \varepsilon'A \exp\bigg(CA\varepsilon'\int_{B+1}^t(\tau+B+1)d\tau\bigg)
\endaligned
$$
When 
$$
\sqrt{CA\varepsilon'} \leq (B+1)^{-1}
$$
and
$$
t\leq \frac{1}{3}(CA\varepsilon')^{-1/2},
$$
one gets that
$$
E^*_g(t,\del^I w_i^{k+1})^{1/2} \leq \varepsilon' A.
$$
Then by an standard method presented in the proof of theorem ... of \cite{So}, 
$$
\lim_{k \rightarrow \infty} w_i^{k} = w_i
$$
is the unique solution of \eqref{appendix eq nonlinear}, and $w_i \in C([0,T(\varepsilon')], H^{s+1}) \cap C^1[0,T(\varepsilon')], H^s)$.
Here one can take 
$$
T(\varepsilon') = C(A\varepsilon)^{-1/2}
$$

To estimate $E_g(B+1,Z^I w_i)$, one takes $\del_t w_i$ as the multiplier and by the standard procedure of energy estimate,
$$
\aligned
E_g(B+1,Z^I w_i) - E^*_g(B+1,Z^I w_i) 
&= \int_{V(B)} \big(Z^I F_i(w, \del w)\del_t w_i - [Z^I,g^{\alpha\beta}\del_{\alpha\beta}] w_i \cdot \del_t w_i\big) dx
\\
&+ \int_{V(B)}\bigg(\del_{\alpha}g^{\alpha\beta}\del_t w_i\del_{\beta}w_i -\frac{1}{2}\del_t g^{\alpha\beta}\del_{\alpha}w_i \del_{\beta} w_i\bigg) dx,
\endaligned
$$
where $V(B) : = \{(t,x): t\geq B+1, t^2-|x|^2 \leq B+1\}\cap \Lambda'$.
When $A\varepsilon'\leq (B+1)^{-2}$, thanks to \eqref{appendix proof decay estimate} and \eqref{appendix proof energy assumption}, 
the right hand side can be controlled by $CA\varepsilon'$. Then one gets
$$
E_g(B+1,Z^I w_i)\leq CA\varepsilon'.
$$
\end{proof}

\begin{center}
\bf{\Large{Acknowledgments}}
\end{center}
Part of this work is motivated by a joint work with the author's doctoral supervisor Prof. Ph. LEFLOCH. The author is grateful to him. 
The author is also grateful to his parents Dr. Qing-jiu MA and Ms. Huiqin-YUAN, his fianc\'ee Miss Yuan-yi YANG and his comrade Ye-ping ZHANG,  
for their successive supports and encouragements.

\end{document}